\documentclass[reqno,10pt]{amsart}
\usepackage{amsmath,amsthm,amssymb,amsfonts,color,,xcolor,pifont,cite}
\usepackage{a4wide}
\usepackage[colorlinks]{hyperref}
\usepackage[showonlyrefs]{mathtools}
\mathtoolsset{showonlyrefs=true}
\usepackage{framed,enumerate}
\usepackage[normalem]{ulem}

\newcounter{corr}
\definecolor{violet}{rgb}{0.580,0.,0.827}
\newcommand{\corr}[3]{\typeout{Warning : a correction remains in page
		\thepage}
	\stepcounter{corr}        
	{\color{blue}\ifmmode\text{\,\sout{\ensuremath{#1}}\,}\else\sout{#1}\fi}
	{\color{red}#2}
	{\color{violet} #3}}

\usepackage{pgf,tikz}
\usepackage{mathrsfs,ulem}
\usetikzlibrary{arrows}

\newtheorem{theorem}{Theorem}[section]
\newtheorem{proposition}[theorem]{Proposition}
\newtheorem{corollary}[theorem]{Corollary}
\newtheorem{lemma}[theorem]{Lemma}

\newtheorem{remark}[theorem]{Remark}
\newtheorem{definition}[theorem]{Definition}

\newcommand{\R}{\mathbb R}
\newcommand{\N}{\mathbb N}

\def\I{\mathcal I}

\def\H{\mathcal H}
\def\eps{\varepsilon}
\def\lt{\left}
\def\rt{\right}
\def\S{\mathcal{S}}
\def\les{\lesssim}
\def\ges{\gtrsim}
\def\Ia{\I_\alpha}
\def\Iae{\I_{\alpha,\eps}}
\def\wE{\widetilde{E}}
\def\wmu{\widetilde{\mu}}
\def\F{\mathcal F}
\def\FaQ{\F_{\alpha,Q}}
\def\FuQ{\F_{1,Q}}
\def\FaQL{\F_{\alpha,Q,\Lambda}}
\def\FaQeL{\F_{\alpha,Q,\Lambda,\eps}}

\def\wF{\widetilde{F}}
\def\one{\boldsymbol{1}}
\def\sIa{I_\alpha}
\def\ha{\widehat{\alpha}}

\author{Michael Goldman}
\address{Universit\'e de Paris and Sorbonne Universit\'e, CNRS, LJLL, F-75005 Paris, France}
\email{michael.goldman@u-paris.fr}

\author{Matteo Novaga}
\address{Department of Mathematics, University of Pisa, 56127 Pisa, Italy} 
\email{matteo.novaga@unipi.it} 

\author{Berardo Ruffini}
\address{Department of Mathematics, University of Bologna, 40126 Bologna, Italy} 
\email{berardo.ruffini@unibo.it}

\numberwithin{equation}{section}

\title[Isoperimetric problem with strong capacitary repulsion]{Rigidity of the ball for an isoperimetric problem with strong capacitary repulsion}

\begin{document}

\begin{abstract}
We consider a variational problem involving  competition between surface tension and charge repulsion. We show that, as opposed to the case of weak (short-range)
interactions where we proved ill-posedness of the problem in a previous paper, 
when the repulsion is stronger  the perimeter dominates the capacitary term at small scales. In particular we prove existence of minimizers for small charges as well as their regularity.
Combining this with the stability of the ball under small $C^{1,\gamma}$ perturbations, this  ultimately leads to the minimality of the ball for small charges.
We cover in particular the borderline case of the $1-$capacity where both terms in the energy are of the same order.
 
\end{abstract}

\maketitle
\tableofcontents

\section{Introduction}
In this paper, we consider a geometric variational problem motivated by  models for charged liquid drops recently studied in a series of papers
\cite{gnrI,gnrII,MurNov,murnovruf,murnovruf2,DHV}. One of the main features of these problems is the strong competition between surface tension and charge repulsion. 
In particular, as opposed to the much studied  Gamow liquid drop model (see  \cite{KnuMu,ChoMuTo}), 
the non-local effects  often dominate the cohesive forces leading to  singular behaviors.    The aim of the paper is to consider the case of very strong short-range repulsion between the charges, thus completing the program started in \cite{gnrI}.

\smallskip

We now introduce the model. Given $\alpha\in(0,N)$ and a measurable set $E\subset\R^N$, we define the Riesz interaction energy 
\begin{equation}\label{def:I}
\I_\alpha(E)=\inf_{\mu(E)=1} \int_{\R^N\times \R^N}\frac{d\mu(x)\,d\mu(y)}{|x-y|^{N-\alpha}}.
\end{equation}
This energy coincides with the inverse of the $\alpha-$capacity, see \eqref{Capalpha}. Letting $P(E)$ denotes the perimeter of $E$ (see \cite{maggi}), we consider for every charge $Q>0$ the functional
\[
 \FaQ(E)=P(E)+Q^2 \Ia(E).
\]
While postponing the discussion about the precise class in which we are minimizing, the aim of the paper is to study for $m>0$ the problem 
\[
\min_{|E|=m}\FaQ(E).
\] 
By  a  scaling argument,  up to renaming the constant $Q$, it is enough to consider the case $m=\omega_N$, where $\omega_N$ is the volume of the unit ball $B_1$.\\

This question is motivated by the model for an electrically charged liquid drop in absence of gravity, introduced by Lord Rayleigh \cite{Ray} in
 the physically relevant case $N=3$ and $\alpha=2$,  and later investigated by many authors (see for instance \cite{Tay,Zel,FonFri,gnrI,MurNov,murnovruf,julin}).
We proved in \cite{gnrI} that, quite surprisingly, for every $N\ge 2$ and $\alpha\in(1,N)$ (in particular  in the Coulombic case $\alpha=2$), the problem is ill-posed. 
Indeed, in that case, we can show that 
\[
 \inf_{|E|=\omega_N, E \textrm{ smooth}} \FaQ(E)= P(B_1).
\]
In words,  starting from smooth sets the lower semicontinuous envelope of the energy $\FaQ$ in $L^1(\R^N)$
reduces to the perimeter. To restore well-posedness of the problem one needs to impose some extra regularity conditions 
such  as bounds on the curvature \cite{gnrI}, entropic terms \cite{MurNov} or the convexity of competitors \cite{gnrII}.  

Later on, it was shown in \cite{murnovruf}  that at least if $N=2$ and $\alpha=1$, the problem admits the ball as unique minimizer as long as $Q$ lies below an explicit threshold, and that no-existence occurs otherwise. 

\vspace{.3cm}

The aim of this paper is  to complement the picture in the case $\alpha\in(0,1]$ for $N\ge 2$. 

\subsection{Main results}
  
As already mentioned above, the first difficulty with this model is to properly define the class of competitors. Indeed, while both the perimeter and $\Ia$ are well-defined in
the class of smooth compact sets, this class does not enjoy good compactness properties. For variational problems involving the perimeter, the usual setup is the one of
sets of finite perimeter (see \cite{maggi}) where we identify two sets $E$ and $F$ if they are equal up to a Lebesgue-negligible set. 
However, it is not hard to see that $\Ia$ is not well-behaved under such identification (we have $\Ia(E)=\Ia(F)$ if $E=F$ outside a set of zero $\alpha-$capacity), see \cite{gnrI}.
As advocated in \cite{murnovruf,murnovruf2} for $N=2$ and $\alpha=1$, we will consider here the class
\begin{equation}\label{def:S}
\S=\left\{E\subset \R^N\,:\, \text{ $E$ is  compact and }\, P(E)=\mathcal H^{N-1}(\partial E)<+\infty  \right\}.
\end{equation} 
We will always identify sets in $\S$ which differ only on a set of Lebesgue measure zero (and thus actually agree $\H^{N-1}$ a.e.), see Remark \ref{rem:welldef}. The variational problem we consider is thus 
\begin{equation}\label{mainprob}
\min_{|E|=\omega_N, E\in \S}\FaQ(E).
\end{equation}
The main result of this paper is the following.

\begin{theorem}\label{thm:main}
	For every $N\ge 2$ and $\alpha\in(0,1]$, there exists $Q_0=Q_0(N,\alpha)>0$ such that for every $Q\le Q_0$, balls are the only minimizers of \eqref{mainprob}.
\end{theorem}

We recall that by \cite{gnrI}, if  $\alpha\in (1,N)$, problem   \eqref{mainprob} does not admit minimizers for any value of $Q>0$. The proof of Theorem \ref{thm:main} 
follows the same general scheme as in \cite{KnuMu} (see also \cite{AcFuMo, f2m3,CaFuPra, MuVes,CanGol} where similar strategies have been used).
Inspired by the proof \cite{CicLeo} of the quantitative isoperimetric inequality, the idea is to prove first  existence of (generalized) minimizers for the problem.
Then, the challenge is to prove regularity estimates for minimizers which are  uniform in $Q$. This allows by compactness to reduce the problem to a second order Taylor
expansion of the energy close to the ball (this is the so-called Fuglede type argument). As we will now see, in our case all three steps present serious difficulties.
Let us point out that when $N=2$ and $\alpha=1$, the proof in \cite{murnovruf} is of totally different nature. Indeed, it uses a combination of convexification and Brunn-Minkowski inequalities.\\
While it could be interesting to see if this argument could be extended to the case $\alpha\in(0,1)$, it is intrinsically limited to $N=2$.
We start with existence of minimizers:

\begin{theorem}\label{thm:introexist}
 For every $N\ge 2$ and $\alpha\in(0,1]$, there exists $Q_1=Q_1(N,\alpha)>0$ such that for every $Q\le Q_1$ minimizers of \eqref{mainprob} exist.
\end{theorem}
This result is proven in Theorem \ref{theo:existence} where we prove actually a bit more. Indeed, we show that if $\alpha<1$,   minimizers  exist for every $Q>0$, 
at least in a generalized sense (see Definition \ref{def:gener}). As in \cite{GolNo,FuEsp,gnrII,CanGol}, a  classical first step is to transform the volume constraint into a penalization,
see Lemma \ref{lemLambda}.   Following for example \cite{GolNo, KnMuNov, FraLieb, novprat,CanGol} we would then like
to prove Theorem \ref{thm:introexist} through a concentration-compactness argument. However, since the class $\S$ is not closed under $L^1$ convergence
and because of the issues related to the lower semi-continuity of $\Ia$ raised above, it is not clear how to argue directly for $\FaQ$. 
The idea is thus to first regularize the functional by  penalizing concentration of the charge. While we believe that the precise choice of regularization is not essential,
in line with \cite{MurNov}, for $\eps>0$, we replace $\Ia$ by 
\[
 \Iae(E)=\min_{\mu(E)=1}\left\{\int_{\R^N\times \R^N}\frac{d\mu(x)\,d\mu(y)}{|x-y|^{N-\alpha}}+\eps\int_{\R^N}\mu^2\,\right\}.
\]
In particular, this functional is well defined in $L^1$ i.e. $\Iae(E)=\Iae(F)$ if $E=F$ a.e. and we can prove in Proposition \ref{prop:genmineps} 
the existence of (generalized) minimizers. \\
In order to conclude the proof of Theorem \ref{thm:introexist} and send $\eps$ to zero, we show in Proposition \ref{prop:density}  that minimizers enjoy density estimates which are 
uniform in $\eps$. This is a consequence of a first almost-minimality property of minimizers proven in Proposition \ref{prop:almostmin1}.
Indeed, using a relatively simple lower bound from \cite{murnovruf} on the Riesz interaction energy of the union of two disjoint sets,  we show that  there exists a constant  $C>0$ such that 
if $E$ is a minimizer for the regularized functional and $F$ is such that $E\Delta F\subset B_r$, then 
\begin{equation}\label{eq:introfirstalmostmin}
 P(E)\le P(F)+ C (Q^2+ r^\alpha) r^{N-\alpha}.
\end{equation}
The desired density bounds then follow from \cite{gnrnote}, see also  \cite{maggi,tamanini} when $\alpha<1$.\\

We then turn to regularity:
\begin{theorem}\label{thm:introreg}
 For every $N\ge 2$ and $\alpha\in(0,1]$, there exist $Q_2=Q_2(N,\alpha)>0$ and $\gamma=\gamma(N,\alpha)\in (0,1/2)$ such that for $Q\le Q_2$ minimizers of \eqref{mainprob} are uniformly (in $Q$) $C^{1,\gamma}$.
\end{theorem}
This result is contained  in Proposition \ref{prop:regsmall} for $\alpha<1$ (see also Remark \ref{rem:regalpha}) and Proposition \ref{prop:regsmall1} for $\alpha=1$. 
On the one hand,  we see from \eqref{eq:introfirstalmostmin} that when $\alpha<1$, we may directly appeal to the classical regularity theory for almost-minimizers of the perimeter, see \cite{tamanini,maggi},
and there is nothing to prove. On the other hand, when $\alpha=1$ the situation is much more delicate. In fact, Proposition \ref{prop:regsmall1} may be seen as one of the main achievements of this paper. 
When $\alpha=1$, while \eqref{eq:introfirstalmostmin} is in general too weak to obtain $C^{1,\gamma}$ regularity, it is still strong enough to yield Reifenberg flatness of $E$ (see Definition \ref{Reifenberg flatness}) when $Q\ll1$  as recently shown in\footnote{as a matter of fact, the present paper served as motivation for \cite{gnrnote}.} \cite{gnrnote}.  In order to improve it to the full $C^{1,\gamma}$ regularity we rely on a second almost-minimality property. 
We show in Proposition \ref{prop:am} that if $\mu_E$ is the optimal charge distribution for $E$ i.e. $\mu_E$ is the minimizer in \eqref{def:I}, and if $E\Delta F\subset B_r$, then 
\begin{equation}\label{eq:introsecondalmostmin}
 P(E)\le P(F)+ C \lt(Q^2\lt(\int_{B_r} \mu_E^{\frac{2N}{N+1}}\rt)^{\frac{N+1}{N}}+ r^N\rt).
\end{equation}
The proof is inspired by \cite[Proposition 4.5]{DHV}. Notice however that we have to deal with difficulties which are quite different from the ones in  \cite{DHV}.
Indeed, on the one hand our  operator is smooth (here it is just the half Laplacian, see \eqref{eq:ELuE}) as opposed to \cite{DHV} where the heart of the problem is the presence of irregular coefficients.
On the other hand in  our case  the charge distribution $\mu_E$ is, a priori, just a measure while in \cite{DHV} it is  known to be   a function in $L^\infty$.
In fact, in light of \eqref{eq:introsecondalmostmin},  the main point here is to prove good integrability properties of the charge distribution $\mu_E$.
This is done in Lemma \ref{lem:keyharm} and Lemma \ref{lem:estmu}
where we prove that for any $\gamma\in (0,1/2)$, if $E$ is a sufficiently flat   Reifenberg sets $E$,
\begin{equation}\label{eq:introReif}
 \lt(\int_{B_r} \mu_E^{\frac{2N}{N+1}}\rt)^{\frac{N+1}{N}}\le C r^{N-1+2\gamma}.
\end{equation}
By comparing this result with the case of the ball, we  can see that this estimate is optimal. It may be seen as an extension to irregular domains of the boundary regularity for the fractional Laplacian developed in \cite{RosSer}.
As in the case of the Laplacian considered in \cite{LemSir}, the main ingredient for the proof is the monotonicity formula of  Alt-Caffarelli-Friedman.
Combining \eqref{eq:introsecondalmostmin} and \eqref{eq:introReif} we find that also for $\alpha=1$, minimizers of \eqref{mainprob} are actually classical almost-minimizers of the perimeter and thus Theorem \ref{thm:introreg} follows. \\

Let us point out that our proof actually applies to a more general class of functionals. Indeed, for $\Lambda\ge 0$ we say that $E$ is a $\Lambda-$minimizer of $\FaQ$ if for every set $F$,
\[
 \FaQ(E)\le \FaQ(F)+\Lambda |E\Delta F|.
\]
Arguing as in Theorem \ref{thm:introreg} we can prove the following result (we state it only for $\alpha=1$ since $\alpha\in(0,1)$ is simpler):
\begin{theorem}\label{thm:cor}
 For every $\gamma\in(0,1/2)$, there is $\bar Q=\bar Q(N,\gamma)>0$ such that for every $Q\le \bar Q$ and $\Lambda\ge 0$, every $\Lambda-$minimizer $E$ of $\F_{1,Q}$ is $C^{1,\gamma}$
 regular outside a singular set $\Sigma$ with $\Sigma=\emptyset$ if $N\le 7$, $\Sigma$ is locally finite if $N=8$ and satisfies $\H^{s}(\Sigma)=0$ is $s> N-8$ and $N\ge 9$.  
\end{theorem}
In particular, this answers a question left open in \cite{murnovruf2}.\\

Thanks to Theorem \ref{thm:introreg} and the quantitative isoperimetric inequality, if $Q$ is small enough then up to translation any minimizer of \eqref{mainprob} is nearly spherical.
By this we mean that fixing $\gamma\in (0,1)$ (which remains implicit), $|E|=\omega_N$, the barycenter of $E$ is in zero and there is $\phi: \partial B_1\mapsto \R$ with
$\|\phi\|_{C^{1,\gamma}(\partial B_1)}\le 1$ such that 
\[
 \partial E=\{(1+\phi(x))x \ : \ x\in \partial B_1\}.
\]
The proof of Theorem \ref{thm:main} is thus concluded once we show  the minimality of the ball among nearly spherical sets:
\begin{theorem}\label{thm:nearlyspherical}
Let $\alpha\in(0,2)$. There exist $Q_3=Q_3(N,\alpha, \gamma)>0$ and $\eps=\eps(N,\alpha, \gamma)>0$ such that for every nearly spherical
set $E$ with $\|\phi\|_{W^{1,\infty}(\partial B_1)}\le \eps$, and every $Q\le Q_3$,
\[
 \FaQ(B_1)\le \FaQ(E).
\]
Moreover, equality is attained only if $E=B_1$.
\end{theorem}
This shows in particular the stability of the ball under small $C^{1,\gamma}$ perturbations if $Q$ is small enough. The counterpart of Theorem \ref{thm:nearlyspherical} for  the Coulomb case $\alpha=2$ has been obtained in \cite[Theorem 1.7]{gnrI}.
The main point of the proof is to show in Proposition \ref{prop:comparison} that 
\[
 \Ia(B_1)-\Ia(E)\le C\lt([\phi]^2_{H^{\frac{\alpha}{2}}(\partial B_1)}+[\phi]_{H^{\frac{2-\alpha}{2}}(\partial B_1)}^2\rt).
\]
The proof of this quantitative estimate for $\Ia$ follows the same general strategy as in \cite{gnrI}. As in \cite{gnrI}, the  difficulty comes from the fact that the optimal measure $\mu_E$ is not explicitly given in terms of $E$. 
There are however some important differences between the Coulomb case $\alpha=2$ and the non-local case $\alpha\in (0,2)$.
 In the Coulomb case, the charges are concentrated on the boundary. 
 This  makes it easier to express the Riesz interaction energy in terms of  $\phi$ with respect to the case $\alpha\in(0,2)$ where  the optimal measure $\mu_E$ has support equal to $E$. 
 Another difficulty here comes from the blow-up of  $\mu_E$ near $\partial E$. 
 \\

The last result of this paper is a nonexistence result in dimension $2$. 
\begin{theorem}\label{thm:nonexintro}
 Let $N=2$ and $\alpha\in(0,1]$. Then, there exists $Q_4=Q_4(N,\alpha)>0$ such that for $Q\ge Q_4$, there are no  minimizers of \eqref{mainprob}. 
\end{theorem}

\medskip 

The  paper is divided into four parts. In Section \ref{sec:prelimininary} we collect the precise notation and definitions used  in the paper. In Section \ref{sec:genmin}, we prove Theorem \ref{thm:introexist} and Theorem \ref{thm:introreg} about existence and regularity of minimizers for \eqref{mainprob}.
We then prove Theorem \ref{thm:nearlyspherical} in Section \ref{sec:ballmin}. A last short section is dedicated to the proof of the nonexistence Theorem \ref{thm:nonexintro}.

\section{Notation}\label{sec:prelimininary}
We will use the notation $A\les B$ to indicate that there exists a constant $C>0$, typically 
depending on the dimension $N$ and on $\alpha$ such that $A\le C B$ (we will specify when $C$ depends on other quantities).  We  write $A\sim B$ if   $A\lesssim B\lesssim A$. 
We use in hypotheses the notation $A\ll B$ to indicate that there exists a (typically small) universal constant $\eps>0$ depending only on $N$ and $\alpha$ such that if $A\le \eps B$ then the conclusion of the statement holds.
\\

For a measurable set  $E\subset \R^N$ and  an open  set $\Omega\subset \R^N$, we denote by $|E|$ the Lebesgue measure of $E$ and  $P(E,\Omega)$ its relative perimeter in $\Omega$. When $\Omega=\R^N$ 
we simply write $P(E,\Omega)=P(E)$ (see \cite{maggi}). We use $\partial E$ to indicate  the topological boundary of a set $E$, $\partial^M E$ the measure-theoretic boundary
and $\partial^* E$ for the reduced boundary.\\

\subsection{Fractional Sobolev spaces, Laplacians and capacities}
We collect here some standard notation and basic properties about fractional Sobolev spaces, Laplacians and capacities.
We refer for instance to \cite{LiebLoss, landkof,DiPaVa} for more information on this topic. For the Fourier transform we use the convention
\[
 \widehat{u}(\xi)=\int_{\R^N} e^{-2i\pi \xi\cdot x} u(x) dx.
\]
For $s\in \R$ we then define the (homogeneous) $H^s$ semi-norm as 
\[
 [u]^2_{H^s(\R^N)}=\int_{\R^N} |\xi|^{2s} |\widehat{u}|^2 d\xi.
\]
When there is no risk of confusion we simply write $[u]_{H^s}$ for $[u]_{H^s(\R^N)}$. We define the $s-$fractional Laplacian by its Fourier transform:
\[
 \widehat{(-\Delta)^s u}=|\xi|^{2s} \widehat{u}
\]
so that by Parseval identity
\begin{equation}\label{Parseval}
 [u]^2_{H^s(\R^N)}=\int_{\R^N} u (-\Delta)^s u.
\end{equation}
For $s\in(0,1)$, there exists $C(N,s)>0$ such that 
\begin{equation}\label{LaplacePV}
 (-\Delta)^s u(x)= C(N,s) \int_{\R^N} \frac{u(x)-u(y)}{|x-y|^{N+2s}}\,dy,
\end{equation}
where the integral is intended in the principal value sense. We then have the alternative formula for the $H^s$ semi-norm
\begin{equation}\label{Hsbis}
 [u]^2_{H^s(\R^N)}=\frac{C(n,s)}{2}\int_{\R^N\times \R^N} \frac{(u(x)-u(y))^2}{|x-y|^{N+2s}} dx dy.
\end{equation}
We will also use fractional Sobolev spaces defined on the unit sphere $\partial B$. For these we take \eqref{Hsbis} as starting point and define for $\phi: \partial B\to \R$ and $s\in (0,1)$,
\begin{equation}\label{HsSphere}
 [\phi]^2_{H^s(\partial B)}=\int_{\partial B\times \partial B} \frac{(\phi(x)-\phi(y))^2}{|x-y|^{N-1+2s}} dx dy.
\end{equation}
Let us point out that we use here a slight abuse of notation and do not distinguish between the volume measure on $\R^N$ and the  one on the sphere.
We recall that for $0<s<s'<1$, if we denote  $\bar \phi=\frac{1}{P(B)}\int_{\partial B} \phi$ we have 
\begin{equation}\label{embed}
 \int_{\partial B}(\phi-\bar \phi)^2\les [\phi]_{H^s(\partial B)}^2\le 2^{2(s'-s)} [\phi]_{H^{s'}(\partial B)}^2\les \int_{\partial B} |\nabla \phi|^2,
\end{equation}
where we write $\nabla \phi$ for the tangential gradient and where the implicit constants depend on $N$, $s$ and $s'$.  Indeed, the first inequality follows from Cauchy-Schwarz as
\begin{multline*}
 \int_{\partial B}(\phi-\bar \phi)^2\le \frac{1}{P^2(B)}\int_{\partial B}\lt(\int_{\partial B} |\phi(x)- \phi(y)| dy\rt)^2 dx\\
 \le \frac{1}{P^2(B)}\int_{\partial B} \lt(\int_{\partial B} \frac{(\phi(x)-\phi(y))^2}{|x-y|^{N-1+2s}} dy\rt)\lt(\int_{\partial B} |x-y|^{N-1+2s} dy\rt) dx\\
 \les \int_{\partial B\times \partial B} \frac{(\phi(x)-\phi(y))^2}{|x-y|^{N-1+2s}} dx dy.
\end{multline*}
The second inequality in \eqref{embed} is immediate while the third can be  deduced by \cite[Proposition 2.4 and Remark 2.8]{DNRV}.

%

For $\alpha\in(0,N)$ and $\mu$ a Radon measure,  we define the Riesz interaction energy as
\[
I_\alpha(\mu)=\int_{\R^N\times\R^N}\frac{d\mu(x)\,d\mu(y)}{|x-y|^{N-\alpha}}.
\]
With this notation, definition \eqref{def:I} becomes
\[
\I_\alpha(E)=\min_{\mu(E)=1}I_\alpha(\mu).
\]

\begin{remark}\label{rem:welldef}
 Recalling the definition \eqref{def:I} of $\S$, we see that for $\alpha\in(0,1]$, $\Ia$ is well defined in $\S$ in the sense that if $E,F\in \S$ with $|E\Delta F|=0$ then actually
 $\H^{N-1}(E\Delta F)=0$ and thus $\Ia(E)=\Ia(F)$ (since $\Ia(E\Delta F)=\infty$, see e.g. \cite[Theorem 8.7]{Mattila}).
\end{remark}

For every $\alpha\in (0,N)$, there exists a constant $C'(N,\alpha)>0$ such that for any radon measure $\mu$,  
 the associated potential 
\[
u(x)=\int_{\R^N} \frac{d\mu(y)}{|x-y|^{N-\alpha}},
\]
satisfies  (see \cite{LiebLoss})
\begin{equation}\label{eq:ELuE}
(-\Delta)^{\frac{\alpha}{2}}u=C'(N,\alpha)\mu\qquad \text{ in $\R^N$}. 
\end{equation}
 In particular, using that 
 \[
  I_\alpha(\mu)=\int_{\R^N} u d\mu=\int_{\R^N} (-\Delta)^{\frac{\alpha}{2}} u (-\Delta)^{-\frac{\alpha}{2}} \mu
  = C'(N,\alpha)\int_{\R^N} \mu (-\Delta)^{-\frac{\alpha}{2}} \mu,
 \]
we have 
 \begin{equation}\label{IH}
  I_\alpha(\mu)=C'(N,\alpha)[\mu]_{H^{-\frac{\alpha}{2}}(\R^N)}^2.
 \end{equation}
Similarly, 
\begin{equation}\label{Iu}
 I_\alpha(\mu)=\frac{1}{C'(N,\alpha)} [u]_{H^{\frac{\alpha}{2}}(\R^N)}^2.
\end{equation}
Moreover, if $E$ is compact and  $\mu_E$ is the equilibrium measure of $E$ i.e. $\Ia(E)=I_\alpha(\mu_E)$, then the corresponding potential $u_E$ satisfies 
\begin{equation}\label{eq:uconstantE} u_E\equiv \Ia(E)  \qquad \text{ on $E$},\end{equation}
see \cite{gnrI,landkof} for a precise justification.
We finally point out that if we define the fractional capacity as 
\[
 C_\alpha(E)= \frac{1}{\Ia(E)},
\]
then it is not hard to check that at least for smooth enough sets $E$, $u_E/\Ia(E)$ is the minimizer of 
\[
 \min_{v\ge \chi_E, v\to 0 \textrm{ at } \infty} [v]_{H^{\frac{\alpha}{2}}(\R^N)}^2 
\]
so that by \eqref{Iu} and \eqref{Hsbis}
\begin{equation}\label{Capalpha}
 C_\alpha(E)= \frac{C(N,\alpha/2)}{C'(N,\alpha)} \inf_{u\in C^\infty_c(\R^N), u\ge \chi_E} \int_{\R^N\times \R^N} \frac{(u(x)-u(y))^2}{|x-y|^{N+2s}} dx dy.
\end{equation}
We refer to \cite{landkof,Mattila} for further information on fractional capacities.

\subsection{Generalized sets and minimizers}
 For a (possibly finite) sequence of sets $\wE=(E^i)_{i\ge 1}$ and measures $\wmu=(\mu^i)_{i\ge 1}$,
 we define

 \begin{equation}\label{defItilde}
  I_{\alpha}(\wmu)=\sum_i I_{\alpha}(\mu^i),
 \end{equation}
\begin{equation}
 \label{deftildeEQ2}
\Ia(\wE)=\inf_{\wmu}\lt\{ \sIa(\wmu) \ : \ \sum_i \mu^i(E^i)=1\rt\} \qquad \textrm{and } \qquad P(\wE)=\sum_i P(E^i).  
\end{equation}
Notice that since $\sIa(\wmu+\wmu')\ge \sIa(\wmu)$, when minimizing over $\wmu$, we may assume without loss of generality that $\mu^i$ is concentrated on $E^i$.

\begin{definition}\label{def:gener}
We call generalized set a collection of sets $\widetilde E=(E^i)_{i\ge1}$ as above, and we set
\begin{equation}\label{def:voltE}
|\widetilde E|=\sum_i|E^i|.
\end{equation}
For $Q>0$, we define the energy of a generalized set as 
 \begin{equation}\label{deftildeEQ}
\FaQ(\wE)=P(\wE)+ Q^2 \Ia(\wE).  
 \end{equation}
We say that $\widetilde E\in \S^\N$ is a (volume constrained) generalized minimizer for $\FaQ$ if, 
for any collection of sets $\widetilde F\in\S^\N$ with $|\widetilde F|=|\widetilde E|$, we have
\[
\FaQ(\widetilde E)\le \FaQ(\widetilde F).
\]
\end{definition}
\section{Existence and  regularity of minimizers}\label{sec:genmin}

 \subsection{Relaxation of the volume constraint}
 For $\Lambda>0$, we relax the volume constraint by considering 
\begin{equation}\label{deftildeEQL}
 \FaQL(\wE)=\FaQ(\wE)+\Lambda\lt| |\widetilde E|-\omega_N\rt|.
\end{equation}
Our first result is that for $\Lambda$ large enough the relaxed problem   coincides with the constrained one.
\begin{lemma}\label{lemLambda}
 For every $\alpha\in(0,N)$, $Q>0$ and every $\Lambda\gg 1+Q^2$, we have 
 \begin{equation}\label{equalrelax}
  \inf_{\wE\in \S^\N}\lt\{ \FaQ(\wE) \ : \ |\widetilde E|=\omega_N\rt\}= \inf_{\wE\in \S^\N} \FaQL(\wE).
 \end{equation}
Moreover, for such $\Lambda$, if  $\wE$ is a minimizer of the right-hand side of \eqref{equalrelax}, then  $|\widetilde E|=\omega_N$.
\end{lemma}
\begin{proof}
 Since the fact that the left-hand side of \eqref{equalrelax} is larger than the right-hand side, it is enough to prove the remaining inequality. 
 Let $\Lambda\gg 1+Q^2$ and assume that there exists $\wE$ with $|\wE|\neq \omega_N$ and 
 \[
  \FaQL(\wE)\le \inf_{\wE\in \S^\N}\lt\{ \FaQ(\wE) \ : \ |\wE|=\omega_N\rt\}.
 \]
Using $B_1$ as competitor we find 
\begin{equation}\label{relax:estimener}
 P(\wE)+Q^2 \Ia(\wE)+\Lambda\lt||\wE|-\omega_N\rt|\les 1+Q^2.
\end{equation}
In particular if $t=\omega_N^{\frac{1}{N}}|\wE|^{-\frac{1}{N}}$, then we can write $t=1+\delta$ with $|\delta|\les \Lambda^{-1}(1+Q^2)\ll1$.
We now use $t\wE=(tE^i)_{i\ge 1}$, which satisfies $ |t\wE |=\omega_N$, as competitor and obtain using Taylor expansion, \eqref{relax:estimener} and $\Lambda\gg 1+ Q^2$ that 
\begin{equation}\label{relax:almost}\Lambda |\delta|\les \delta \lt( (N-1) P(\wE)- (N-\alpha) Q^2\Ia(\wE)\rt).\end{equation}
Now if $\delta\ge 0$, this implies
\[
 \Lambda \delta \les \delta P(\wE)\stackrel{\eqref{relax:estimener}}{\les} \delta (1+Q^2)
\]
and thus $\Lambda\les 1+Q^2$ which is a contradiction with the hypothesis $\Lambda\gg 1+Q^2$. In the case $\delta\le 0$ we reach the same contradiction since \eqref{relax:almost} yields this time
\[
 \Lambda |\delta|\les |\delta| Q^2\Ia(\wE)\stackrel{\eqref{relax:estimener}}{\les} |\delta| (1+Q^2).
\]

\end{proof}
 \subsection{The regularized functional}
 As mentioned in the introduction, since  the capacitary term $\Ia$ is not well defined in $L^1$, which is the natural setting to minimize the perimeter, we will first show existence of (generalized)
 minimizers for a regularized energy. 
 For $\eps>0$ and a positive measure $\mu$ we define
 \begin{equation}\label{eq:defIepsmu}
  I_{\alpha,\eps}(\mu)=I_{\alpha}(\mu)+\eps\int_{\R^N} \mu^2=\int_{\R^N\times \R^N} \frac{d\mu(x)d\mu(y)}{|x-y|^{N-\alpha}}+\eps\int_{\R^N} \mu^2
 \end{equation}
with the understanding that $\Iae(\mu)=\infty$ if $\mu\notin L^2(\R^N)$. We then define  for $\wmu=(\mu^i)_{i\ge 1}$ and $\wE=(E^i)_{i\ge 1}$, in analogy with \eqref{defItilde} and \eqref{deftildeEQ2},
\begin{equation}\label{defIaetildeE}
 I_{\alpha,\eps}(\wmu)=\sum_i I_{\alpha,\eps}(\mu^i) \qquad \textrm{and } \qquad \Iae(\wE)=\inf_{\wmu}\lt\{ I_{\alpha,\eps}(\wmu) \ : \ \sum_i \mu^i(E^i)=1\rt\}.
\end{equation}
\begin{remark}
 If $E$ and $F$ are two measurable sets with $|E\Delta F|=0$ then $\Iae(E)=\Iae(F)$. Indeed, every measure $\mu$ with $I_{\alpha,\eps}(\mu)<\infty$ is in $L^2$ and thus satisfies $\mu(E)=\mu(F)$.
 When considering $\Iae$ instead of $\Ia$ we can therefore identify sets which agree Lebesgue a.e.
\end{remark}

\begin{lemma}\label{rieszbis}
Let $\wE=(E_i)_{i\ge 1}$ be a generalized set with $|\wE|\in(0,\infty)$, then (recall definition \eqref{def:voltE})
\begin{equation}\label{eq:rieszbis}
\frac{\eps}{|\wE|}\le \Iae(\wE)\,\le\, \frac{c(N,\alpha)}{|\wE|^\frac{N-\alpha}{N}} + \frac{\eps}{|\wE|}\,,
\end{equation}
where
$$
c(N,\alpha) = \int_{B_1\times B_1} \frac{1}{|x-y|^{N-\alpha}}\,dxdy.
$$
\end{lemma}

\begin{proof}
We start with the upper bound. Let $m=|\wE|$. and for every $i\ge 1$ let $B^i$ be a ball such that $|B^i|=|E^i|$. Choosing $\mu^i = \chi_{E^i}/m$ in the definition of $\Iae(\wE)$ 
and recalling the Riesz rearrangement inequality, we get
\begin{eqnarray*}
\Iae(\wE)&\le& \frac{1}{m^2}\sum_i \int_{E^i\times E^i} \frac{dx dy}{|x-y|^{N-\alpha}} + \frac{\eps}{m}
\\
&\le&  \frac{1}{m^2}\sum_i \int_{B^i\times B^i} \frac{dx dy}{|x-y|^{N-\alpha}} + \frac{\eps}{m}\\
&=& \frac{c(N,\alpha)}{m^2}\sum_i |E^i|^{1+\frac{\alpha}{N}} + \frac{\eps}{m}\\
&\le &\frac{c(N,\alpha)}{m^\frac{N-\alpha}{N}} + \frac{\eps}{m}\,.
\end{eqnarray*}
To obtain the lower bound we simply observe that for every $\wmu=(\mu_i)_{i\ge 1}$ with $\sum_i \mu^i(E^i)=1$, we have by Cauchy-Schwarz
\begin{equation}\label{eq:abovebis}
 \eps^{\frac{1}{2}}=\eps^{\frac{1}{2}}\sum_i \mu^i(E^i)\le \lt(\sum_{i} |E^i|\rt)^{\frac{1}{2}} \lt(\eps \sum_i\int_{E^i} (\mu^i)^2\rt)^{\frac{1}{2}}\le m^{\frac{1}{2}}  \Iae^{\frac{1}{2}}(\wmu).
\end{equation}
The desired bound follows by minimizing in $\wmu$.
\end{proof}
As a consequence, of Lemma \ref{rieszbis}, we can prove the existence of an optimal measure for $\Iae(\wE)$.
\begin{corollary}\label{cor:number}
 For every $\eps>0$ and every generalized set $\wE=(E^i)_{i\ge 1}$ with $|\wE|<\infty$ and $\Iae(\wE)<\infty$, there exists a unique optimal measure $\wmu$ for $\Iae(\wE)$.
\end{corollary}
\begin{proof}
 Uniqueness follows from strict convexity of the energy so we only focus on the existence part of the statement. We first notice that from the definition \eqref{defIaetildeE} of $\Iae$ we have 
 \begin{equation}\label{prob:qi}
  \Iae(\wE)=\inf\lt\{ \sum_i q_i^2 \Iae(E^i) \ : \ \sum_i q_i=1\rt\}.
 \end{equation}
Hence, the existence of an optimal $\wmu$ follows if we can prove that on the one hand, for every fixed set $E$ of finite volume, there exists an optimal measure for $\Iae(E)$
and on the other hand that there exists an optimal distribution of charges $(q_i)_{i\ge 1}$ for \eqref{prob:qi}. \\
We thus start by considering a fixed set $E$ with $|E|+\Iae(E)<\infty$ and prove the existence of an optimal charge $\mu$.
If $\mu_n$ is a minimizing sequence, arguing as in \eqref{eq:abovebis} we find that for every $R>0$,
\[
 \mu_n(B_R^c)\le \eps^{-\frac{1}{2}} |E\cap B_R^c|^{\frac{1}{2}} \Iae^{\frac{1}{2}}(\mu_n).
\]
Therefore $\mu_n$ is tight and we can extract a sequence converging weakly in $L^2(\R^N)$ to a measure $\mu$ with $\mu(E)=1$. By lower semi-continuity of $\Iae$ (see \cite[(1.4.5)]{landkof}), $\mu$ is a minimizer for $\Iae(E)$.\\
We now turn to the existence of an optimal charge distribution $(q_i)_{i\ge 1}$. For this we first observe that from the first inequality in \eqref{eq:rieszbis}, 
\[
 \sum_i \Iae^{-1}(E^i)\le \frac{1}{\eps} \sum_i |E^i|<\infty
\]
and thus in particular $\lim_{I\to \infty} \sum_{i\ge I}  \Iae^{-1}(E^i)=0$. Now for every $(q_i)_{i\ge 1}$ and every $I\in \N$, by Cauchy-Schwarz
\[
 \sum_{i\ge I} q_i\le \lt(\sum_{i\ge I} q_i^2 \Iae(E^i)\rt)^{\frac{1}{2}}\lt(\sum_{i\ge I} \Iae^{-1}(E^i)\rt)^{\frac{1}{2}}
\]
so that tightness of minimizing sequences follows leading to the  existence of an optimal distribution $(q_i)_{i\ge 1}$.
\end{proof}
In order to prove that generalized minimizers are almost minimizers of the perimeter, we will need the following lemma which is  adapted from  \cite[Lemma 2]{murnovruf} (see also  \cite[Lemma 13]{murnovruf2}). 

\begin{lemma}\label{cutbis}
For every generalized set $\wE=(E\cup F)\times (E^i)_{i\ge 2}$ with $E$ and $F$ sets of positive measure such that  $|E\cap F|=0$, if we define $\widetilde{F}=F\times(E^i)_{i\ge 2}$
 we have
\begin{equation}\label{stimottabis}
\Iae(\wE)\ge \Iae(\widetilde{F})- \frac{\Iae(\widetilde{F})^2}{\Iae(E)}.
\end{equation}
\end{lemma}
\begin{proof}
We  first show  that
\begin{equation}\label{stimathetabis}
\Iae(\wE)\ge \min_{\theta\in [0,1]} \theta^2 \Iae(E) +  (1-\theta)^2 \Iae(\widetilde{F}). 
\end{equation}
 Let $\wmu=(\mu^i)_{i\ge 1}$ be optimal for $\Iae(\wE)$. We may assume without loss of generality that $\mu^1(E)\neq 0$ and $\mu^1(F)+\sum_{i\ge 2} \mu^i(E^i)\neq 0$
 since in the first case we would have $\Iae(\wE)=\Iae(\widetilde{F})$ and in the second case $\Iae(\wE)=\Iae(E)$ which both imply \eqref{stimottabis}. We now define 
 \[
  \mu= \frac{\mu^1|_{E}}{\mu^1(E)}, \quad \nu^1=\frac{\mu^1|_{F}}{1-\mu^1(E)}, \quad \textrm{ and } \quad \nu^i=\frac{\mu^i}{1-\mu^1(E)}, \ \forall i\ge 2.
 \]
With this definition, $\mu$ is admissible for $\Iae(E)$ and $\widetilde{\nu}=(\nu^i)_{i\ge 1}$ is admissible for $\Iae(\widetilde{F})$ and we have 
\begin{multline*}
 \Iae(\mu^1)\ge (\mu^1(E))^2 \lt(\int_{E\times E} \frac{d\mu(x)d\mu(y)}{|x-y|^{N-\alpha}}+\eps\int_E \mu^2\rt)\\
  + (1-\mu^1(E))^2\lt(\int_{F\times F} \frac{d\nu^1(x)d\nu^1(y)}{|x-y|^{N-\alpha}}+\eps\int_F (\nu^1)^2\rt) \\
 =  (\mu^1(E))^2 \Iae(\mu)+ (1-\mu^1(E))^2\Iae(\nu^1)
\end{multline*}
so that by definition \eqref{defIaetildeE} of $\Iae(\widetilde{\nu})$,
\[
 \Iae(\wE)\ge (\mu^1(E))^2 \Iae(\mu)+ (1-\mu^1(E))^2\Iae(\widetilde{\nu}) \ge (\mu^1(E))^2 \Iae(E)+ (1-\mu^1(E))^2\Iae(\widetilde{F}). 
\]
This proves \eqref{stimathetabis}. Optimizing in $\theta$  together with the inequality $(1+t)^{-1}\ge 1-t$ for $t\ge 0$ yields  \eqref{stimottabis}.
\end{proof}

 \subsection{Existence of generalized minimizers for the regularized energy}
 In line with \eqref{deftildeEQL}, we introduce the regularized energy
 \begin{equation}\label{deftildeEQLeps}
  \FaQeL(\wE)=P(\wE)+Q^2\Iae(\wE)+\Lambda \lt||\wE|-\omega_N\rt|.
 \end{equation}
The aim of this section is to prove the existence of minimizers for this functional. 
We start with the simple  observation that for $\FaQeL$, minimizing among classical or generalized sets gives the same infimum energy.

\begin{lemma}\label{lem:equalinf}
	For every $\alpha \in (0,N)$, $Q$, $\Lambda$, $\eps>0$, we have
	\begin{equation}\label{eq:equalinf}
	\inf\left\{ \FaQeL(E)\right\}=\inf\left\{  \FaQeL(\wE)\right\}.
		\end{equation}
	\end{lemma}
\begin{proof}
	Since the left-hand side of \eqref{eq:equalinf} is larger than the right-hand side, it is enough to prove that for every $\delta>0$ and  every generalized set $\wE=(E^i)_{i\ge 1}$,
	we can construct a set $E$ with $\FaQeL(E)\le \FaQeL(\wE)+\delta$.  For $I\in \N$ and $R>0$, let $F^i=E^i\cap B_R$ if $i\le I$, $F^i=\emptyset$ otherwise and 
	set $\widetilde{F}= (F^i)_{i\ge 1}$. We first observe that for each fixed $i$, $\lim_{R\to \infty} |E^i\cap B_R|=|E^i|$. Combining this with the fact that  $\sum_i |E^i|<\infty$ we see that we
	can choose $I$ and $R$ large enough so that 
	\begin{equation}\label{eq:Lambdasmall}
	 \Lambda \lt|\sum_{i=1}^I |F^i|-\omega_N\rt|\le  \Lambda \lt||\wE|-\omega_N\rt| +\delta.
	\end{equation}
Moreover, thanks to the co-area formula we may further assume that 
\begin{equation}\label{eq:persmall}
	 \sum_{i=1}^I P(F^i)\le  P(\widetilde{E}) +\delta.
	\end{equation}
We now turn to the last term in the energy. Let $\wmu=(\mu^i)_{i\ge 1}$ be the optimal measure for $\Iae(\wE)$ given by Corollary \ref{cor:number}.
We then set for $i\le I$, $\nu^i=\frac{\mu^i|_{F^i}}{\sum_{i=1}^I \mu^i(F^i)}$ and $\nu^i=0$ otherwise so that $\widetilde{\nu}=(\nu^i)_{i\ge 1}$ is admissible for $\Iae(\widetilde{F})$. 
Using that $\sum_{i=1}^I \mu^i(F^i)$ converges to $1$ as $I\to \infty$ and $R\to \infty$, we can also assume that $I$ and $R$ are chosen such that in addition to \eqref{eq:Lambdasmall} and \eqref{eq:persmall} we have
\begin{equation}\label{eq:Ismall}
 Q^2\Iae(\widetilde{\nu})= \frac{Q^2}{(\sum_{i=1}^I \mu^i(F^i))^2}\lt(\sum_{i=1}^I \int_{F^i\times F^i} \frac{d\mu^i(x)d\mu^i(y)}{|x-y|^{N-\alpha}} +\eps \int_{F^i} (\mu^i)^2\rt)\le Q^2\Iae(\wE)+\delta.
\end{equation}
We finally choose for every $i\le I$ a point $x^i\in \R^N$ such that $\min_{i\neq j} |x^i-x^j|\gg R$ and define 
\[
 E=\cup_{i=1}^I (F^i+x^i) \qquad \textrm{and } \qquad \nu(x)=\sum_{i=1}^I \nu^i(x-x^i).
\]
Since $F^i\subset B_R$ by construction, the sets $F^i+x^i$ are pairwise disjoint and from \eqref{eq:Lambdasmall} and \eqref{eq:persmall} we have 
\[
 P(E)+\Lambda \lt| |E|-\omega^N\rt|=\sum_{i=1}^I P(E^i)+\Lambda \lt|\sum_{i=1}^I |F^i|-\omega_N\rt| \le  P(\widetilde{E}) + \Lambda \lt||\wE|-\omega_N\rt|+2\delta.
\]
Finally we observe that $\nu$ is admissible for $\Iae(E)$ with 
\begin{multline*}
 Q^2\Iae(\nu)=Q^2\Iae(\widetilde{\nu}) +Q^2 \sum_{i\neq j} \int_{F^i\times F^j} \frac{d\nu^i(x) d\nu^j(y)}{|x-y|^{N-\alpha}}\\
 \le Q^2\Iae(\widetilde{\nu}) 
 +\frac{Q^2}{\min_{i\neq j} |x^i-x^j|^{N-\alpha}}\stackrel{\eqref{eq:Ismall}}{\le} Q^2\Iae(\widetilde{\nu}) +2\delta
\end{multline*}
provided $\min_{i\neq j} |x^i-x^j|$ is large enough. Since $\Iae(E)\le \Iae(\nu)$,  we find as anticipated that 
\[
 \FaQeL(E)\le \FaQeL(\wE)+4\delta.
\]

	\end{proof}

We can now prove the existence of generalized minimizers for $\FaQeL$. This will be proven by a concentration-compactness argument which relies on isoperimetric effects to avoid 
the loss of mass together with the semi-continuity of $\Iae$ with respect to $L^1_{loc}$ convergence. 
This type of arguments is relatively standard  by now (see for instance  \cite{CanGol} which we closely follow or \cite{GolNo, KnMuNov, FraLieb, novprat}). 
However, we face here the additional difficulty that we need to avoid not only loss of volume but also loss of charge in the limit.

\begin{proposition}\label{prop:genmineps}
	For every $\alpha \in (0,1]$, $Q>0$, $\eps>0$ and  $\Lambda\gg 1+Q^2$,  generalized minimizers of $\FaQeL$ exist.
\end{proposition}

\begin{proof}
	Let $(E_n)_{n\ge 1}$ be a (classical)  minimizing sequence for $\FaQeL$. By Lemma \ref{lem:equalinf} it is also a minimizing sequence in the class of generalized sets. 
	Using for instance the ball $B_1$ as competitor we have $\sup_n \FaQeL(E_n)\les 1+Q^2$.
	In particular, if we let  $m_n=|E_n|$, up to extraction $m_n\to m\in (0,\infty)$. Fix $L\gg m^{\frac{1}{N}}$ and 
	consider  a partition of $\R^N$ into cubes $(Q_{i,n})_{i\ge 1}$ where $Q_{i,n}=[0,L]^N+z_i$, with $z_{i,n}\in (L\mathbb Z)^N$. We let $m_{i,n}=|E_n\cap Q_{i,n}|$
	and assume without loss of generality that for every $n$, $m_{i,n}$ is decreasing in $i$.
	Moreover we tacitly consider from now on only the indices $i$ such that $m_{i,n}>0$. We let $\mu_n$ be the optimal measure for $\Iae(E_n)$ and set $q_{i,n}=\mu_n(Q_{i,n})$.\\
	
	We start by proving tightness of $(m_{i,n})_{i\ge 1}$ and $(q_{i,n})_{i\ge 1}$. For $m_{i,n}$, we argue as usual that thanks to the relative isoperimetric 
	inequality (recall that with our choice of $L$, $|Q_{i,n}\cap E_n|\le |Q_{i,n}|/2$)
	\[
	 \sum_{i} m^{\frac{N-1}{N}}_{i,n}\les \sum_i P(E_n,Q_{i,n})=P(E_n)\les 1+Q^2.
	\]
Using that $m_{i,n}\le \frac{m}{i}$ we conclude that for every $I\in \N$,
\begin{equation}\label{eq:tightper}
 \sum_{i\ge I} m_{i,n}\le \lt(\frac{m}{I}\rt)^{\frac{1}{N}} \sum_{i\ge I} m_{i,n}^{\frac{N-1}{N}}\les (1+Q^2)\lt(\frac{m}{I}\rt)^{\frac{1}{N}}.
\end{equation}
For $q_{i,n}$ we argue as in \eqref{eq:abovebis} and obtain invoking twice Cauchy-Schwarz,
\[
  \sum_{i\ge I} q_{i,n}\le \sum_{i\ge I} m_{i,n}^{\frac{1}{2}}\lt(\int_{E_n\cap Q_{i,n}} \mu_n^2\rt)^{\frac{1}{2}}
  \le \lt(\sum_{i\ge I} m_{i,n}\rt)^{\frac{1}{2}}\lt(\int_{\R^N} \mu_n^2\rt)^{\frac{1}{2}}
  \stackrel{\eqref{eq:tightper}}{\les}\eps^{-\frac{1}{2}}(1+Q^2)\lt(\frac{m}{I}\rt)^{\frac{1}{2N}}.
\]
Therefore, up to extraction we have $\lim_{n\to \infty} m_{i,n}=m_i$ with $\sum_i m_i=m$ and $\lim_{n\to \infty} q_{i,n}=q_i$ with $\sum_i q_i=1$.\\
	
	We now construct a generalized set $\wE$ which will be our generalized minimizer. By the perimeter bound, up to extraction we have for every $i$,  
	 $E_n-z_{i,n}\to E^i$  in $L^1_{\rm loc}$ for some sets $E^i$. Moreover $\mu_n^i=\mu_n(\cdot +z_{i,n})$ converges weakly in $L^2$ to some $\mu^i$. We can further assume that for every $i,j$, $|z_{i,n}-z_{j,n}|\to a_{ij}\in [0,\infty]$. We now say that $i\sim j$ if $a_{ij}<\infty$ and denote by $[i]$ the equivalence class of $i$. 
	Notice that if $i\sim j$ then  $E^i$ and $E^j$ are translated of each  other. For each class of equivalence we denote 
	\[
	 m_{[i]}=\sum_{j\sim i} m_j \qquad \textrm{and } \qquad q_{[i]}=\sum_{j\sim i} q_j
	\]
so that we have $\sum_{[i]} m_{[i]}=m$ and $\sum_{[i]} q_{[i]}=1$. For every $i$, using the convergence of $E_n-z_{i,n}$ to $E^i$ and of $\mu_{n}^i$ to $\mu^i$, and the definition of the equivalence relation, we have 
\[
 |E^i|=m_{[i]} \qquad \textrm{and } \qquad \mu^i(E^i)=q_{[i]}.
\]
Up to relabeling, we may now assume that each class of equivalence $[i]$ is made of a single element. If we set $\wE=(E^i)_{i\ge 1}$ and $\wmu=(\mu^i)_{i\ge 1}$ we have just shown that $\wmu$ is admissible for $\Iae(\wE)$. Let us finally prove that 
\begin{equation}\label{eq:toproveexistence}
  P(\wE) +Q^2\Iae(\wmu)+ \Lambda \lt||\wE|-\omega_N\rt|\le \liminf_{n\to \infty} P(E_n)+Q^2\Iae(\mu_n)+ \Lambda\lt| |E_n|-\omega_N\rt|.
\end{equation}
We consider separately each term of the energy. Since $|\wE|=m=\lim_{n\to \infty} |E_n|$, the volume term is not a problem. For the first term, we fix $I\in \N$ and $R>0$. If $n$ is large enough, we can assume that for $i,j\le I$ with $i\neq j$,
$|z_{i,n}-z_{j,n}|\gg R$. By the co-area formula we can find for every $i\le I$ a radius $R_n\in (R,2R)$ such that 
\[
 \sum_{i\le I} \H^{N-1}(\partial B_{R_n}(z_{i,n})\cap E_n)\les \frac{1}{R}. 
\]
If $E^{i,R_n}=(E_n-z_{i,n})\cap B_{R_n}$, we thus have 
\[
 \sum_{i\le I} P(E^{i,R_n})\le P(E_n) +\frac{C}{R}.
\]
From this bound we see that $E^{i,R_n}$ converges in $L^1_{\rm loc}$ to a set $E^{i,R}$ which itself converges to $E^i$ as $R\to \infty$. We thus have
\[
 \sum_{i\le I} P(E^i)\le \sum_{i\le I} \liminf_{R\to \infty} P(E^{i,R})\le \sum_{i\le I} \liminf_{R\to \infty}\liminf_{n\to \infty}  P(E^{i,R_n})\le\liminf_{n\to \infty} P(E_n).
\]
For the last term, we use similarly that for every fixed $I\in \N$ and $R>0$,
\[
 \sum_{i\le I} \Iae(\mu^i|_{B_R})\le \liminf_{n\to \infty} \sum_{i\le I} \Iae(\mu_n^i|_{B_R})\le \liminf_{n\to \infty} \Iae\lt(\sum_{i\le I} \mu^n|_{B_R(z_{i,n})}\rt)\le \liminf_{n\to \infty} \Iae(\mu_n).
\]

\end{proof}
 \subsection{First almost minimality property and existence of minimizers for the original problem}
 In this section we use Lemma \ref{cutbis} to prove a first almost minimality property for generalized minimizers of $\FaQeL$. In order to pass to the limit $\eps\to 0$ it is crucial that the estimates are uniform in $\eps$.
 \begin{proposition}\label{prop:almostmin1}
There exists $C>0$ depending only on $N$ and $\alpha\in(0,N)$ with the following property.  For every $Q>0$, $\eps\in(0,1)$ and $\Lambda \sim 1+Q^2$ for which Lemma \ref{lemLambda} applies,  every generalized minimizer $\wE=(E^i)_{i\ge 1}$ of $\FaQeL$ is an almost minimizer of the perimeter in the sense that for every $i\ge 1$, $x\in \R^N$ and $r\ll1$,
  \begin{equation}\label{eq:almostmin}
   P(E^i)\le P(F) +C\lt(Q^2 +r^\alpha\rt) r^{N-\alpha} \qquad \forall   F\Delta E^i\subset B_r(x).
  \end{equation}

 \end{proposition}
\begin{proof}
 Without loss of generality we may assume that $i=1$ and $x=0$. To simplify a bit notation we denote $E=E^1$. Using $\widetilde{F}=F\times (E^i)_{i\ge2}$ as competitor and the minimality of $\wE$ we have after simplifications
 \begin{equation}\label{eq:quasimininter}
  P(E)\le P(F)+ Q^2\lt(\Iae(\widetilde{F})-\Iae(\wE)\rt)+ \Lambda|E\Delta F|.
 \end{equation}
Since $P(E\cap F)+P(E\cup F)\le P(E)+P(F)$, it is enough to prove \eqref{eq:almostmin} under the additional condition $E\subset F$ or $F\subset E$. If $E\subset F$ then $\Iae(\wE)\ge \Iae(\widetilde{F})$ and thus \eqref{eq:almostmin} follows from $|E\Delta F|\les r^N$. \\
We are left with the case $F\subset E$. Writing $E=F\cup (E\backslash F)$ and appealing to \eqref{stimottabis} from Lemma \ref{cutbis}, we have 
\begin{equation}\label{eq:lowerI}
 \Iae(\widetilde{F})-\Iae(\wE)\le \frac{\Iae^2(\widetilde{F})}{\Iae(E\backslash F)}.
\end{equation}
Now on the one hand, since by Lemma \ref{lemLambda}, $|E|+\sum_{i\ge 2} |E^i|=\omega_N$, $|F|+\sum_{i\ge 2} |E^i|=\omega_N- |E\backslash F|\ges 1$ (recall that $r\ll1$) and thus by \eqref{eq:rieszbis} of Lemma \ref{rieszbis},
\[
 \Iae(\widetilde{F})\les 1.
\]
On the other hand, since $E\backslash F\subset B_r$ we have
\begin{equation}\label{eq:lowerboundIae}\Iae(E\backslash F)\ge \Iae(B_r)+\eps\inf_{\mu(B_r)=1} \int_{B_r} \mu^2\ges r^{-(N-\alpha)} +\eps r^{-N}\ge r^{-(N-\alpha)}.\end{equation}
Putting these two things together, \eqref{eq:lowerI} yields
\[
  \Iae(\widetilde{F})-\Iae(\wE)\les r^{N-\alpha}.
\]
Plugging  this in combination with $|E\Delta F|\les r^N$ in \eqref{eq:quasimininter} concludes the proof of \eqref{eq:almostmin}. 
\end{proof}
\begin{remark}
 From \eqref{eq:lowerboundIae} we see that we can improve \eqref{eq:almostmin} to 
 \[
   P(E^i)\le P(F) +C\lt(Q^2\min(r^{-\alpha}, \eps^{-1}) +1\rt) r^N \qquad \forall   F\Delta E^i\subset B_r(x).
 \]
This means that for every $\alpha\in(0,N)$, if  $r\le \eps^{1/\alpha}$ then the classical regularity theory for perimeter almost-minimizers applies (see \cite{maggi}). 
In particular, for $\alpha=2$, this gives a very elementary proof of the regularity of minimizers for the functional considered in \cite{MurNov, DHV}
if the permittivity of the droplet is assumed to coincide with the permittivity of the vacuum (see however \cite[Remark 4.6]{DHV} where it is observed that this assumption
would also simplify their proof).
\end{remark}

At this point we see the difference between the case $\alpha>1$ and $\alpha\le 1$. Indeed, in the latter case,
thanks to  \eqref{eq:almostmin}, we may appeal to the regularity theory for almost-minimizers of the perimeter (since $N-\alpha\ge N-1$). 
We start with the simpler part which consists of the density estimates. Since the cases $\alpha<1$ and $\alpha=1$ are treated  differently,
we introduce the notation $\one_{\alpha=1}=1$ if $\alpha=1$ and $\one_{\alpha=1}=\infty$ if $\alpha\in(0,1)$.
\begin{proposition}\label{prop:density}
 For every $\alpha\in(0,1]$ and  $Q>0$ let $\Lambda\sim 1+Q^2$ be such that Proposition \ref{prop:almostmin1} applies. Then, for every  $\eps\in(0,1]$, and every generalized minimizer $\wE=(E^i)_{i\ge 1}$ of $\FaQeL$,
 if $\max(Q^2r^{1-\alpha},r)\ll1$ and $x\in \partial^M E^i$ (recall that $\partial^M$ is the measure-theoretical boundary), 
 \begin{equation}\label{eq:densvol}
  \min(|E^i\cap B_r(x)|, |B_r(x)\backslash E^i|)\ges r^N
 \end{equation}
and 
\begin{equation}\label{eq:densper}
  P(E^i,B_r(x))\sim r^{N-1}.
\end{equation}
  As a consequence, there exists $Q_1>0$ such that for $Q\le \bar Q\le Q_1 \one_{\alpha=1}$, up to the choice of a representative,
  every generalized minimizer is made of  finitely many $E^i$, each of which is  connected  with $E^i\in \S$ and for which $\partial E^i=\partial^M E^i$. Moreover, the number of such components as well as their diameter depends only on $\bar Q$.
\end{proposition}

\begin{proof}
 Estimates \eqref{eq:densvol} and \eqref{eq:densper} follow directly from \cite[Proposition 3.1]{gnrnote}.
 For $\alpha<1$ they can also be obtained (under slightly stronger hypothesis on $r$) from  the more classical theory, see for instance \cite{maggi}.\\
 The  regularity of the minimizers as well as the bound on the number and diameter of the  connected components is classical (see e.g. \cite{KnuMu}) once we observe that for every $Q\le  \bar Q\le Q_1\one_{\alpha=1}$ there is $\bar r$ depending only on $\bar Q$ such that $\max(Q^2{\bar r}^{1-\alpha},\bar r)\ll1$ for every $Q\le \bar Q$. The fact that we may assume that each component $E^i$ of $\wE$ is made of a single connected component follows from $\Iae(E\cup F)\ge \Iae(E\times F)$ for every sets $E$, $F$ with $E\cap F=\emptyset$.
\end{proof}
Before stating  the full conclusions of the regularity theory for perimeter almost minimizers, let us conclude the proof of the existence of generalized volume-constrained minimizers of $\FaQ$.

\begin{theorem}\label{theo:existence}
 Let $Q_1$ be given by Proposition \ref{prop:density}. Then for every $0<Q\le \bar Q\le Q_1 \one_{\alpha=1}$ there exist 
  generalized  minimizers  $\wE=(E^i)_{i=1}^I\in \mathcal S^\N$ of 
 \[
  \min_{\wE\in \S^\N}\lt\{ \FaQ(\wE) \  : \ |\wE|=\omega_N\,\rt\}.
 \]
Moreover, for each $i\le I$, $E^i$ is a perimeter almost minimizer in the sense of \eqref{eq:almostmin} and both $I$ and ${\rm diam}(E^i)$ are bounded by a constant depending only on $\bar Q$.
\end{theorem}
\begin{proof}
 Let $\Lambda\sim 1+\bar Q^2$ be such that both Lemma \ref{lemLambda} and Proposition \ref{prop:genmineps} apply. By the latter, for every $\eps\in(0,1]$ and $Q\le \bar Q$,
 there exists a generalized minimizer $\wE_\eps$ of $\FaQeL$. Moreover, by Proposition \eqref{prop:density}, $\wE_\eps=(E_\eps^i)_{i=1}^I$ for some connected  sets $E^i\in \S$,
 where  both  $I$  and their diameters depend only on $\bar Q$. Thanks to the uniform  density bounds \eqref{eq:densvol} and \eqref{eq:densper}
 we can extract a subsequence for which $E_\eps^i$ converges both in $L^1$ and in the Kuratowski sense to some $E^i\in \S$. By compactness  of perimeter almost minimizers, see \cite{gnrnote}, $E^i$ satisfies \eqref{eq:almostmin}. We set $\wE=(E^i)_{i\le I}$. Using that $\Iae(\wE_\eps)\ge \Ia(\wE_\eps)$ and the fact that $\Ia$ is lower semi-continuous under this convergence (see e.g. \cite[Theorem 4.2]{gnrI}), we obtain
 \[
  \liminf_{\eps\to 0} \FaQeL(\wE_\eps)\ge \FaQL(\wE).
 \]
We now prove that 
\[
 \inf_{\wF\in \S^\N}\lt\{\FaQL(\wF)\rt\}\ge \limsup_{\eps\to 0} \inf_{\wF\in \S^\N}\lt\{\FaQeL(\wF)\rt\},
\]
which combined with the previous inequality would show that $\wE$ is a generalized minimizer of $\FaQL$ as
\[
\FaQL(\wE)\ge\inf_{\wF\in \S^\N}\lt\{\FaQL(\wF)\rt\}. 
\]
Arguing exactly as in Lemma \ref{lem:equalinf}, we see that it is enough to prove that for every  $F\in \S$, there exists a sequence $F_\eps$ such that 
\begin{equation}\label{eq:toproveexist}
 \limsup_{\eps\to 0} \FaQeL(F_\eps)\le \FaQL(F). 
\end{equation}
By \cite{Sc} applied to $F^c$, we can find  smooth compact sets  $F_\delta$ with $F\subset F_\delta$, $P(F^\delta)\le P(F)+ \delta$ and $| |F|-|F^\delta||\le \delta$. Since $\Ia(F)\ge \Ia(F^\delta)$ as $F\subset F^\delta$ we have (actually there is equality) 
\[
 \limsup_{\delta\to 0} \FaQL(F^\delta)\le \FaQL(F)
\]
and we can thus further assume that $F$ is smooth in the proof of \eqref{eq:toproveexist}. For smooth sets, by \cite[Proposition  2.16]{gnrI}\footnote{The statement of \cite[Proposition 2.16]{gnrI} requires $F$ to be connected but the proof works for disconnected sets as well.}, we can find for every $\delta>0$ a function $f_\delta\in L^\infty(F)$ with $\int_F f_\delta=1$ and such that 
\[
 \Ia(f_\delta)\le \Ia(F)+\delta.
\]
Since for every $\delta>0$, $\lim_{\eps\to 0} \Iae(f_\delta)=\Ia(f_\delta)$,  a diagonal argument shows that $\Ia(F)=\lim_{\eps\to 0} \Iae(F)$. Using $F_\eps=F$ we conclude the proof of \eqref{eq:toproveexist}.\\

As $\wE$ is a generalized minimizer of $\FaQL$, Lemma \ref{lemLambda} implies that $|\wE|=\omega_N$ and thus $\wE$ is also a volume-constrained generalized minimizer of $\FaQ$. 
\end{proof}
We end this section by recalling the regularity properties of generalized minimizers and show in particular that for small charge $Q$ they are actually classical minimizers. We start with the case $\alpha<1$.
\begin{proposition}\label{prop:regsmall}
 For  $\alpha\in(0,1)$ and $Q>0$ let $\wE=(E^i)_{i=1}^I$ be a volume-constrained generalized minimizer of $\FaQ$. Then, $\partial^* E^i$ (recall that  $\partial^*$ denotes the reduced boundary) 
 are $C^{1,\frac{1}{2}(1-\alpha)}$ regular. Moreover, denoting by $\Sigma_i=\partial E^i\backslash \partial^* E^i$,  we have that for every $i$,  $\Sigma_i$ is empty if $N\le7$ is at most finite if $N=8$ and satisfies $\H^s(\Sigma_i)=0$ if $s> N-8$ and $N\ge 9$.\\
 In addition, for $Q\ll1$, $\wE=E_Q$ is a classical volume-constrained minimizer of $\FaQ$, $\Sigma(E_Q)=\emptyset$  and for every $\beta< \frac{1}{2}(1-\alpha)$, $E_Q$ converges in $C^{1,\beta}$ to $B_1$ as $Q\to 0$. 
\end{proposition}
\begin{proof}
 The conclusion follows from the classical regularity theory for perimeter almost minimizers, see \cite{tamanini,maggi} and the fact that by the quantitative isoperimetric inequality, up to translation and relabeling,
 \[
  \lt(|E^1\Delta B_1|+ \sum_{i\ge 2} |E^i|\rt)^2\les P(\wE)-P(B_1)\le Q^2 \Ia(B_1), 
 \]
which implies in conjunction with \eqref{eq:densvol} that for $Q$ small enough, $E^i=\emptyset$ for $i\ge 2$ (so that $\wE=E^1$ is a classical minimizer) together with  the convergence to $B_1$.
\end{proof}
For $\alpha=1$ it is well-known that in general \eqref{eq:almostmin} does not even imply  $C^1$ regularity. In order to state the counterpart of Proposition \ref{prop:regsmall} in this case, let us first recall the definition of Reifenberg flat sets.
\begin{definition}\label{Reifenberg flatness}
 Let $\delta,r_0>0$ and $x\in \R^N$. We say that $E$ is $(\delta,r_0)-$Reifenberg flat in $B_{r_0}(x)$ if for every $B_{r}(y)\subset B_{r_0}(x)$, there exists an hyperplane $H_{y,r}$
 containing $y$ and such that 
 \begin{itemize}
  \item we have 
   \[
  \frac{1}{r} d(\partial E\cap B_{r}(y), H_{y,r}\cap B_{r}(y)) \le \delta,
 \]
 where $d$ denotes the Hausdorff distance;
 \item one of the connected components of 
 \[
 \{ d(\cdot, H_{y,r})\ge  2 \delta r\}\cap B_{r}(y)
 \]
is included in $E$ and the other in $E^c$.
 \end{itemize}
We say that $E$ is uniformly $(\delta,r_0)-$Reifenberg flat if the above conditions hold for every $x\in \partial E$.
\end{definition}

\begin{proposition}\label{prop:reg1}
 Let $\alpha=1$. There exists $Q_2>0$ such that for every $Q\le Q_2$, every volume-constrained generalized minimizer of $\FaQ$ is a classical minimizer. Moreover, for every $\delta>0$, there exist $Q_\delta, r_\delta>0$ such that for every $Q\le Q_\delta$, every  volume-constrained minimizer $E_Q$ of $\FaQ$ is uniformly $(\delta,r_\delta)-$Reifenberg flat and up to translation,
 \begin{equation}\label{eq:closeness}
  |E_Q\Delta B_1|^2\les Q^2.
 \end{equation}

\end{proposition}
\begin{proof}
 The proof is exactly as for Proposition \ref{prop:regsmall}, replacing the classical regularity theory by \cite[Corollary 1.4]{gnrnote}.
\end{proof}

 \subsection{Second almost minimality property and regularity of minimizers}
 The aim of this section is to prove that in the case $\alpha=1$, we can pass from the Reifenberg flatness of volume-constrained minimizers of $\FaQ$ stated in Proposition \ref{prop:reg1} 
 to almost  $C^{1,\frac{1}{2}}$ regularity. This will be obtained by proving a second almost minimality property for minimizers together with a higher  integrability result
 for the optimal measure $\mu$.
 \begin{remark}\label{rem:regalpha}
  Let us point out that  using a similar proof for $\alpha\in(0,1)$,  it would be possible to improve the $C^{1,\frac{1}{2}(1-\alpha)}$
  regularity from Proposition \ref{prop:regsmall} to almost $C^{1,\frac{1}{2}}$.
  However, in this case, the proof of the integrability of $\mu$ can be greatly simplified by appealing directly to \cite{RosSer} (see also \eqref{estimmu} below).
  Moreover, we expect that any $C^{1,\beta}$ regularity may be improved to higher regularity through the Euler-Lagrange equation, see \cite{murnovruf2}.   
 \end{remark}
We start with the quasi-minimality property.
\begin{proposition}\label{prop:am}
 There exists $C>0$ depending only on $N$ with the following property.  If $Q\le 1$ and   $E$ is a volume-constrained minimizer of $\FuQ$ with $\mu_E$
 the corresponding $1/2$-harmonic measure i.e. $\I_1(E)=I_1(\mu_E)$,
 then for every $x\in \R^N$ and $0<r\ll1$,
\[
 P(E)\le P(F)+ C\lt( Q^2 \lt(\int_{B_r(x)} \mu_E^{\frac{2N}{N+1}}\rt)^{\frac{N+1}{N}}+ r^N\rt) \qquad \forall E\Delta F\subset B_r(x).
\]
\end{proposition}

\begin{proof}
Without loss of generality we may assume that $x=0$ and $\mu_E\in L^{\frac{2N}{N+1}}(B_r)$ since otherwise there is nothing to prove. 
By Lemma \ref{lemLambda}, there exists a universal constant $\Lambda>0$ (recall that $Q\le 1$) such that $E$ is a minimizer of 
\[
 \FuQ(E)+\Lambda||E|-\omega_N|.
\]
Arguing as in the proof of Proposition \ref{prop:almostmin1}, we see that it is enough to prove that for every  $F\subset E$ with $E\backslash F\subset B_r$, 
\begin{equation}\label{eq:toprovealmostmin2}
 \I_1(F)\le \I_1(E)+C \lt(\int_{E\backslash F} \mu_E^{\frac{2N}{N+1}}\rt)^{\frac{N+1}{N}}.
\end{equation}

In order to prove \eqref{eq:toprovealmostmin2} we follow the general strategy of \cite[Proposition 4.5]{DHV}  and use
\[
 \mu=\left(\mu_E+ \frac{\mu_E(E\backslash F)}{|F|}\right)\chi_F 
\]
as a competitor for $\I_1(F)$. We define
\[
 u_E(x)=\int_{E} \frac{d\mu_E(y)}{|x-y|^{N-1}} \qquad \textrm{and} \qquad u(x)=\int_{E}\frac{d\mu(y)}{|x-y|^{N-1}}
\]
the potentials associated to $\mu_E$ and $\mu$.
We recall from \eqref{eq:ELuE} that $u$ solves on $\R^N$ the equation
\[
(-\Delta)^{\frac{1}{2}}u=C'(N,1) \mu,
\]
and that by \eqref{Iu},
\[
\frac{1}{C'(N,1)}[u]_{H^{\frac{1}{2}}}^2=\int_{E} u\,d\mu=I_1(\mu).
\]
Let us notice that since $u_E=\I_1(E)$ on $E$, recall \eqref{eq:uconstantE}, and since $\mu_E(E)=\mu(E)=1$,
\begin{equation}\label{simplify}
 \int_E u_E d(\mu-\mu_E)=0.
\end{equation}
Since $\I_1(F)\le I_1(\mu)$ and $\I_1(E)=I_1(\mu_E)$ we have 
\[
 \I_1(F)-\I_1(E)\le \int_E u d\mu -\int_E u_Ed\mu_E= \int_E (u-u_E) d(\mu-\mu_E)+ \int_E u d\mu_E -\int_E u_E d\mu_E.
\]
Using Fubini we have 
\[
  \int_E u d\mu_E =\int_E  u_E d\mu\stackrel{\eqref{simplify}}{=} \int_E u_E d\mu_E
\]
and we get 
\[
 \I_1(F)-\I_1(E)\le\int_E (u-u_E) d(\mu-\mu_E)=\frac{1}{C'(N,1)} [u-u_E]^2_{H^{\frac{1}{2}}}.
\]
We now estimate $[u-u_E]^2_{H^{\frac{1}{2}}}$. For this, using H\"older inequality and Sobolev embedding we write
\begin{align*}
 [u-u_E]^2_{H^{\frac{1}{2}}}&=\int_E (u-u_E) d(\mu-\mu_E)\\
 &\le \|u-u_E\|_{L^{\frac{2N}{N-1}}} \|\mu-\mu_E\|_{L^{\frac{2N}{N+1}}}\\
 &\les [u-u_E]_{H^{\frac{1}{2}}}\|\mu-\mu_E\|_{L^{\frac{2N}{N+1}}}.
\end{align*}
Using Young inequality, this leads to 
\[
 \I_1(F)-\I_1(E)\les [u-u_E]^2_{H^{\frac{1}{2}}}\les \|\mu-\mu_E\|_{L^{\frac{2N}{N+1}}}^2.
\]
We are left with estimating $\|\mu-\mu_E\|_{L^{\frac{2N}{N+1}}}$. By definition of $\mu$, we have  $\mu-\mu_E= \frac{\mu_E(E\backslash F)}{|F|}\chi_F- \mu_E \chi_{E\backslash F}$ and thus  
\[
\begin{aligned}
  \|\mu-\mu_E\|_{L^{\frac{2N}{N+1}}}^2&= \left(\int_E \left| \frac{\mu_E(E\setminus F)}{|F|}\chi_F -\mu_E\chi_{E\setminus F } \right|^{\frac{2N}{N+1}} \right)^{\frac{N+1}{N}}\\
  &= \left( \int_F \left(\frac{\mu_E(E\setminus F)}{|F|}\right)^{\frac{2N}{N+1}}    + \int_{E\setminus F} \mu_E^{\frac{2N}{N+1}}\right)^{\frac{N+1}{N}}\\  
  &  \les \frac{\mu_E(E\backslash F)^2}{|F|^{\frac{N-1}{N}}}+ \lt(\int_{E\backslash F} \mu_E^{\frac{2N}{N+1}}\right)^{\frac{N+1}{N}}\\
  &\stackrel{|F|\ges 1}{\les} \mu_E(E\backslash F)^2 + \lt(\int_{E\backslash F} \mu_E^{\frac{2N}{N+1}}\right)^{\frac{N+1}{N}}.
\end{aligned}
\]
Finally, by H\"older inequality,
\[
 \mu_E(E\backslash F)^2\le \lt(\int_{E\backslash F} \mu_E^{\frac{2N}{N+1}}\rt)^{\frac{N+1}{N}} |E\backslash F|^{\frac{N-1}{N}}\les \lt(\int_{E\backslash F} \mu_E^{\frac{2N}{N+1}}\rt)^{\frac{N+1}{N}}.
\]
This concludes the proof of \eqref{eq:toprovealmostmin2}.
\end{proof}
 From Proposition \ref{prop:am}, we see that in order to prove that $E$ is a perimeter almost-minimizer in the classical sense, it is enough to show decay estimates
 for $\|\mu_E\|_{L^{2N/(N+1)}(B_r(x))}$ for $x\in \partial E$.
We start by proving  the following  H\"older estimate for the potentials.

\begin{lemma}\label{lem:keyharm}
 For every $\delta>0$, there exists $\gamma\in (0,\frac{1}{2})$ with $\gamma\to \frac{1}{2}$ as $\delta\to 0$ such that if $E$ is a bounded $(\delta,r_0)-$Reifenberg flat domain then 
 \begin{equation}\label{eq:mainharm}
  |1-\I_1^{-1}(E)u_E|\les \frac{d^{\gamma}(\cdot,\partial E)}{r_0^\gamma},
 \end{equation}
 where $u_E(x)=\int_{E} \frac{d\mu_E(y)}{|x-y|^{N-1}}$ and $\mu_E$ is such that $\I_1(E)=I_1(\mu_E)$.
\end{lemma}

\begin{proof}
 By scaling we may assume that $r_0=1$. We follow ideas from the proofs of \cite{LemSir,zilio} and use the Alt-Caffarelli-Friedman monotonicity formula to show \eqref{eq:mainharm}.
 Let $u=1-\I_1^{-1}(E)u_E$ and $v$ be the harmonic extension of $u$ to $\R^{N+1}_+$. Since $u\in[0,1]$, also $v\in[0,1]$. Notice that since $u\le 1$ it is enough to prove \eqref{eq:mainharm} in  $\{d(\cdot,\partial E)\ll 1\}$. For every $x\in \R^{N+1}_+$ and every $r>0$, 
 we let $B_r^+(x)= B_r(x)\cap \R^{N+1}_+$ and  $\partial^+ B_r(x)=\partial B_r(x)\cap \R^{N+1}_+$.  We claim that 
 \begin{equation}\label{eq:aimholder}
  \frac{1}{r^{N-1}}\int_{B_r^+(x)}|\nabla v|^2\les r^{2\gamma}\int_{B_1^+(x)}\frac{|\nabla v|^2 }{|x-y|^{N-1}} \qquad \forall 0<r\ll1 
 \end{equation}
for some exponent $\gamma>0$ with $\gamma\to \frac{1}{2}$ as $\delta\to 0$
and 
\begin{equation}\label{eq:claimholder}
 \sup_{\R^{N+1}_+}\int_{B_1^+(x)}\frac{|\nabla v|^2 }{|x-y|^{N-1}}\les 1.
\end{equation}
Provided \eqref{eq:aimholder} and \eqref{eq:claimholder}  hold, we can conclude the proof of \eqref{eq:mainharm} using Poincar\'e inequality, Campanato's criterion and $v=0$ in $E\times \{0\}$. Eventually we show that $\gamma\to1/2$ as $\delta\to0$. We devote a step for each of these three claims\\
{\bf Step 1:} Proof of \eqref{eq:aimholder}.
To show that \eqref{eq:aimholder} holds,  we first observe that it is enough to consider $x\in \partial E\times\{0\}$. 
Indeed, assume the statement is proven in that case. Then  for $x\notin \partial E\times\{0\}$,  using either an  odd reflection or an even reflection with respect to $x_{N+1}=0$
we may assume that $v$ is harmonic in $B_r(x)$ for every $r\le \bar r=\min(1,d(x,\partial E\times \{0\}))$.
It is then a classical fact that
\[
 r\mapsto \frac{1}{r^{N+1}}\int_{B_r(x)}|\nabla v|^2
\]
is increasing (this follows for instance from sub-harmonicity of $|\nabla v|^2$ which is itself a consequence of Bochner formula). Therefore for any $0\le \gamma\le 1$,
\[
 \frac{1}{r^{N-1}}\int_{B_r^+(x)}|\nabla v|^2\le \lt(\frac{r}{\bar r}\rt)^2\frac{1}{\bar r^{N-1}}\int_{B_{\bar r}^+(x)}|\nabla v|^2\le \lt(\frac{r}{\bar r}\rt)^{2\gamma}\frac{1}{\bar r^{N-1}}\int_{B_{\bar r}^+(x)}|\nabla v|^2.
\]
If $\bar r\ll 1$  then \eqref{eq:aimholder} follows from the case $ x\in \partial E\times\{0\}$. If instead $\bar r\ges 1$,  
\[
 \lt(\frac{r}{\bar r}\rt)^{2\gamma}\frac{1}{\bar r^{N-1}}\int_{B_{\bar r}^+(x)}|\nabla v|^2\les r^{2\gamma}\int_{B_{\bar r}^+(x)}\frac{|\nabla v|^2}{|x-y|^{N-1}}\le r^{2\gamma}\int_{B_1^+(x)}\frac{|\nabla v|^2 }{|x-y|^{N-1}},
\]
which proves \eqref{eq:aimholder} also in this case.\\

Let now $x\in \partial E\times\{0\}$. Without loss of generality we may assume that $x=0$. For every $r>0$, let\footnote{we denote by $\nabla_\tau$ the tangential gradient on the sphere and $\partial_\nu$ the normal derivative} 
\[
 \lambda(r)=\min\lt\{\frac{\int_{\partial^+ B_r} |\nabla_\tau v|^2}{\int_{\partial^+ B_r} v^2} : \ v=0 \textrm{ on } E\times\{0\}\cap \partial B_r^+\rt\}
\]
be the first eigenvalue of the Laplacian on the half-sphere with Dirichlet boundary conditions on $E$. 
Define then the function
\[
 \gamma(\lambda)=\sqrt{\lt(\frac{N-1}{2}\rt)^2+\lambda}-\frac{N-1}{2}
\]
and then 
\[
  \overline{\gamma}=\min_{r\le 1} [\gamma(r^2\lambda(r))].
\]
We claim that for $r\in(0,1]$, the function
\[
 \Phi(r)=\frac{1}{r^{2 \overline{\gamma}}}\int_{B_r^+}\frac{|\nabla v|^2}{|x|^{N-1}}
\]
is increasing.
For this we follow almost verbatim the proof of \cite[Theorem 2.6]{zilio}  (see also \cite[Lemma 2.10 \& Lemma 2.11]{zilio}). 
In particular, a regularization argument is required to make rigorous all the computations below but we  refer the reader to \cite{zilio} for the details.
Computing the logarithmic derivative of $\Phi$, we have  
\[
 \frac{\Phi'}{\Phi}=-2\frac{ \overline{\gamma}}{r}+\lt(\int_{\partial^+ B_r} \frac{|\nabla v|^2}{|x|^{N-1}}\rt)\lt(\int_{B_r^+}\frac{|\nabla v|^2}{|x|^{N-1}}\rt)^{-1}
\]
and it is therefore enough to prove that 
\begin{equation}\label{eq:toprovemono}
 \lt(\int_{\partial^+ B_r} \frac{|\nabla v|^2}{|x|^{N-1}}\rt)\lt(\int_{B_r^+}\frac{|\nabla v|^2}{|x|^{N-1}}\rt)^{-1}\ge 2\frac{ \overline{\gamma}}{r}.
\end{equation}
We first claim that 
\begin{equation}\label{firstclaim}
 \int_{B_r^+} \frac{|\nabla v|^2}{|x|^{N-1}} dx \le \frac{1}{r^{N-1}}\int_{\partial^+B_r} v \partial_\nu v +\frac{N-1}{2 r^N}\int_{\partial^+ B_r} v^2.
\end{equation}
For this we notice that since $\Gamma=|x|^{1-N}$ is the Green function of the Laplacian on $\R^{N+1}$, we have $\Delta\Gamma\le 0$ and moreover, since it is radially symmetric, we have $\partial_{N+1} \Gamma=0$
if $x_{N+1}=0$. Using integration by parts we have (using that on $\partial B_r^+\cap\{x_{N+1}=0\}$, $ v\partial_\nu v=0$)
\begin{align*}
 \int_{B_r^+} |\nabla v|^2\Gamma &=\int_{\partial B_r^+} v \Gamma \partial_\nu v-\int_{B_r^+} \frac{1}{2}\nabla (v^2)\cdot \nabla \Gamma\\
 &=\int_{\partial^+ B_r} \Gamma v \partial_\nu v-\int_{\partial B_{r}^+} \frac{1}{2} v^2 \partial_\nu \Gamma +\frac{1}{2} \int_{B_+^+} v^2\Delta \Gamma\\
 &\le \int_{\partial^+ B_r} \Gamma v \partial_\nu v-\int_{\partial^+ B_{r}} \frac{1}{2} v^2 \partial_\nu \Gamma \\
 &= \frac{1}{r^{N-1}}\int_{\partial^+B_r} v \partial_\nu v +\frac{N-1}{2 r^N}\int_{\partial^+ B_r} v^2.
\end{align*}
This proves \eqref{firstclaim}. We thus have 
\begin{align*}
\lefteqn{\lt(\int_{\partial^+ B_r} \frac{|\nabla v|^2}{|x|^{N-1}}\rt)\lt(\int_{B_r^+}\frac{|\nabla v|^2}{|x|^{N-1}}\rt)^{-1}}\\
&=\lt(r^{1-N}\int_{\partial^+ B_r}|\nabla v|^2\rt)\lt(\int_{B_r^+}|\nabla v|^2 \Gamma\rt)^{-1}\\
  &\stackrel{\eqref{firstclaim}}{\ge}\lt(r^{1-N}\int_{\partial^+ B_r}|\nabla v|^2\rt)\lt(r^{1-N}\lt(\int_{\partial^+B_r} v \partial_\nu v +\frac{N-1}{2 r}\int_{\partial^+ B_r} v^2\rt)\rt)^{-1}\\
  &\ge \lt(\int_{\partial^+ B_r}|\nabla_\tau v|^2 +\int_{\partial^+ B_r}|\partial_\nu v|^2\rt)\lt(\lt(\int_{\partial^+B_r} v^2\rt)^{\frac{1}{2}}\lt( \int_{\partial^+B_r} (\partial_\nu v)^2\rt)^{\frac{1}{2}} +\frac{N-1}{2 r}\int_{\partial^+ B_r} v^2\rt)^{-1}\\
   &= \lt( \frac{\int_{\partial^+ B_r}|\nabla_\tau v|^2}{\int_{\partial^+B_r} v^2} +\frac{\int_{\partial^+ B_r}|\partial_\nu v|^2}{\int_{\partial^+B_r} v^2}\rt)\lt(\lt(\frac{\int_{\partial^+ B_r}|\partial_\nu v|^2}{\int_{\partial^+B_r} v^2}\rt)^{\frac{1}{2}}+\frac{N-1}{2 r}\rt)^{-1}\\
   &\ge \min_{t>0} \frac{\lambda(r) +t^2}{t+ \frac{N-1}{2 r}}.
\end{align*}
A direct computation shows that the above minimum is attained for $t_{min}=\frac{1}{r}\gamma(r^2\lambda(r))$ and that
$\min_{t>0} \frac{\lambda(r) +t^2}{t+ \frac{N-1}{2 r}}=2t_{min}=\frac{2}{r}\gamma(r^2\lambda(r))$ so that eventually
\[
 \lt(\int_{\partial^+ B_r} \frac{|\nabla v|^2}{|x|^{N-1}}\rt)\lt(\int_{B_r^+}\frac{|\nabla v|^2}{|x|^{N-1}}\rt)\ge \frac{2}{r}\gamma(r^2\lambda(r))\ge \frac{2}{r}  \overline{\gamma}.
\]
This concludes the proof of \eqref{eq:toprovemono}. By monotonicity of $\Phi$ we have 
\[
 \frac{1}{r^{2 \overline{\gamma}+N-1}}\int_{B_r^+}|\nabla v|^2\le \Phi(r)\le \Phi(1)=\int_{B_1^+}\frac{|\nabla v|^2}{|x|^{N-1}}
\]
and the proof of \eqref{eq:aimholder} with $\gamma= \overline{\gamma}$ is completed.\\

{\bf Step 2:} Proof of  \eqref{eq:claimholder}. For $R\gg1$ we have 
\begin{align*}
 \int_{B_1^+(x)}\frac{|\nabla v|^2 }{|x-y|^{N-1}}&\le \int_{B_{R}^+(x)}\frac{|\nabla v|^2 }{|x-y|^{N-1}}\\
 &\stackrel{\eqref{firstclaim}}{\les} \frac{1}{R^{N-1}}\int_{\partial^+ B_{R}(x)} |v||\partial_\nu v|+ \frac{1}{R^N}\int_{\partial^+ B_{R}(x)} v^2\\
 &\stackrel{|v|\le 1}{\les} \frac{1}{R^{N-1}}\int_{\partial^+ B_{R}(x)} |\partial_\nu v|+ 1.
%
\end{align*}
Since $v(z)=1-\I^{-1}_1(E)\int_E\frac{d\mu(y)}{|z-y|^{N-1}}$ and since $E$ is bounded, if $R$ is large enough and $z\in \partial^+ B_{R}(x)$,
\[
 |\nabla v(z)|\les \frac{\I^{-1}_1(E)}{|z|^N}
\]
and thus 
\[
 \frac{1}{R^{N-1}}\int_{\partial^+ B_{R}(x)} |\partial_\nu v|\les  \frac{\I^{-1}_1(E)}{R^{N-1}}.
\]
Sending $R\to \infty$, we conclude the proof of \eqref{eq:claimholder}.\\

\noindent
{\bf Step 3:} Asymptotic on $\bar \gamma$ and conclusion. We finally show that $ \overline{\gamma}\to 1/2$ as $\delta\to 0$.  Since $\gamma(\lambda)$ is an increasing function of $\lambda$ and since for every $r>0$, $\lambda(r)$ 
is monotone under inclusion (i.e. if we make the dependence in $E$ explicit, then $F\subset E$ implies $\lambda_F(r)\le \lambda_E(r)$), 
 it is enough to prove that $ \overline{\gamma}_{H_\delta}(r)\to \frac{1}{2}$ where 
\[
 H_{\delta}=\{x_1\le -\delta\}.
\]
If $\delta=0$, then $ \overline{\gamma}_{H_0}(r)=\frac{1}{2}$, see \cite[Proposition 2.12]{zilio}. Since $ \overline{\gamma}_{H_\delta}(r)$ 
does not depend on $r$, it is enough to consider $r=1$ and drop the dependence in $r$.\\
The proof is then concluded observing that $\delta\to \lambda_{H_{\delta}}$ is continuous as $\delta\to 0$. Indeed, this can be proven by an easy $\Gamma-$convergence argument.
If $u_\delta$ is a minimizer for $\lambda_{H_{\delta}}$, then up to normalization we may assume that $\int_{\partial^+B_1} u_\delta^2=1$ so that $u_\delta$ is bounded in $H^1(\partial^+ B_1)$ and its trace on $H_\delta$ is bounded in $H^{1/2}$.
Therefore, up to extraction it converges weakly in $H^1( \partial^+B_1)$ to a function $u_0$ which vanishes on $H_0$ (by compact embedding of $H^{1/2}$ in $L^2$ for instance). Therefore $u_0$ is admissible for $\lambda_{H_0}$ and we have 
\[
 \lambda_{H_0}\ge \liminf_{\delta\to 0} \lambda_{H_\delta}\ge  \liminf_{\delta\to 0} \int_{\partial^+ B_1}|\nabla_\tau u_\delta|^2\ge \int_{\partial^+ B_1} |\nabla_\tau u_0|^2\ge \lambda_{H_0}.
\]


\end{proof}
We now convert estimate \eqref{eq:mainharm} on the potential into the desired statement on $\mu_E$.
\begin{lemma}\label{lem:estmu}
For every $\gamma\in (0,\frac{1}{2})$, there exists $\delta_0>0$  such that for every $r_0>0$ and every  $(\delta,r_0)-$Reifenberg flat domain $E$ with $\delta\le \delta_0$,
$\mu_E\in L^{\frac{2N}{N+1}}_{\rm loc}(\R^N)$ and
 for every $x\in \R^N$ and $r<r_0/2$ there holds
\begin{equation}\label{toprovemu}
\lt(\int_{B_r(x)} \mu_E^{\frac{2N}{N+1}}\rt)^{\frac{N+1}{N}}\les r^{N-1+2\gamma},
\end{equation}
where the implicit constant depends on $N$, $\gamma$,  $r_0$ and $|E|$.
\end{lemma}


\begin{proof} 
Let $\gamma=\gamma(\delta)$ be given by Lemma \ref{lem:keyharm}. We first derive from  \eqref{eq:mainharm} the following estimate on $\mu_E$:
\begin{equation}\label{eq:mainmu}
 \mu_E\les d^{-(1-\gamma)}(\cdot,\partial E).
\end{equation}
If $u_E$ denotes the associated potential,
\begin{align*}
 C'(N,1)\mu_E(x)&\stackrel{\eqref{eq:ELuE}}{=}(-\Delta)^{\frac{1}{2}}u_E(x)\stackrel{\eqref{LaplacePV}}{=}C(N,1/2)\int_{E^c}\frac{\mathcal I_1(E)-u_E(y)}{|x-y|^{N+1}}dy\\
 &\stackrel{\eqref{eq:mainharm}}{\les}\frac{\mathcal I_1(E)}{r_0^\gamma} \int_{E^c}\frac{ d^{\gamma}(y,\partial E)}{|x-y|^{N+1}}dy\les \frac{\mathcal I_1(E)}{r_0^\gamma}\int_{B_{d(x,\partial E)}^c(x)} \frac{dz}{|z|^{N+1-\gamma}}\\
 &\les \frac{\mathcal I_1(E)}{r_0^\gamma} d^{-(1-\gamma)}(x,\partial E)\les d^{-(1-\gamma)}(x,\partial E),
\end{align*}
 where in the last line we used that if $B$ is a ball of measure $|E|$ then $\I_1(E)\le\I_1(B)$. This follows for instance  from  the fractional Polya-Sz\"ego inequality \cite{FS}
 and the capacitary definition \eqref{Capalpha} of $\I_1$ (see also \cite{Betsakos}).\\

\noindent We now  prove \eqref{toprovemu}. We may assume without loss of generality that $x=0$ and  $|\gamma-\frac{1}{2}|\ll 1$. 
For $P>0$, we set  $\mu_P=\min\{\mu_E,P\}$. Clearly $\mu_P$ is an integrable function and $\mu_P\to\mu$ a.e. in $E$. Moreover, since $0\le\mu_P\le\mu$, it satisfies inequality \eqref{eq:mainmu}.
We first claim that there exist $C_0,C_1>0$ such that for every $x\in \partial E$ and every $r\le r_0/2$, there exists a set $A(x)\subset \partial E$ such that 
\begin{equation}\label{hypA}
 \sharp A(x)\le C_1 \delta^{1-N}
\end{equation}
and 
\begin{equation}\label{conclmu}
 \int_{B_r(x)} \mu_P^{\frac{2N}{N+1}}\le C_0 r^{N-\frac{2N}{N+1}(1-\gamma)}+\sum_{y\in A(x)} \int_{ B_{5\delta r}(y)}\mu_P^{\frac{2N}{N+1}}.
\end{equation}
Again, there is no loss of generality by restricting ourselves to $x=0$. Recall that by Definition \ref{Reifenberg flatness},  since $E$ is $(\delta,r_0)-$Reifenberg flat, for  every $r\le r_0/2$,
there exists an hyperplane $H_{r}$ such that  
\[
d(\partial E\cap {B_r}, H_{r}\cap B_r)\le \delta r.
\]
In particular, if $N_{r}=\{y\in B_r \ : \ d(y,H_{r})>2 \delta r\}$ we have for $y\in N_{r}$, $d(y,\partial E)\sim d(y, H_{r})$ so that 
\begin{equation}\label{rec1}
\begin{aligned}
 \int_{B_r} \mu_P^{\frac{2N}{N+1}}
 &\le \int_{N_{r}} \mu_P^{\frac{2N}{N+1}} + \int_{B_r\cap N^c_{r}} \mu_P^{\frac{2N}{N+1}} \\
  &\stackrel{\eqref{eq:mainmu}}{\le} C r^{N-1} \int_{2\delta r}^{r} \frac{dt}{t^{\frac{2N}{N+1}(1-\gamma)}} + \int_{B_r\cap N^c_{r}} \mu_P^{\frac{2N}{N+1}}\\
& \le C_0 r^{N-\frac{2N}{N+1}(1-\gamma)}+ \int_{B_r\cap N^c_{r}} \mu_P^{\frac{2N}{N+1}}.
\end{aligned}
\end{equation}

We now estimate the last term on the right-hand side. By triangle inequality, for every $x\in N_r^c\cap B_r$, $d(x, \partial E\cap B_r)\le 3 \delta r$ and thus setting  $r_1=5\delta r$
we have that $\{B_{r_1}(y)\}_{y\in\partial E\cap B_r}$ is a covering of $N_{r}^c\cap B_r$.  By  Vitali covering Lemma we can extract a finite subset of points $A\subset \partial E\cap B_r$
such that 
\begin{itemize}
	\item $\left\{B_{r_1/5}(y) \right\}_{y\in A}$ is made up of pairwise disjoint balls, and
	\item $\left\{B_{r_1}(y) \right\}_{y\in A}$ is still a covering of $N_{r}^c\cap B_r$.
\end{itemize}
Since for $y\in A$, $B_{r_1/5}(y)=B_{ \delta r}(y)\subset N_r^c\cap B_{(1+\delta)r}$, the first condition gives 
\[
r_1^N  \sharp A \lesssim |N_{r}^c\cap B_{(1+\delta)r}|\sim (\delta r)r^{N-1}
\]
which, by  definition of $r_1$ yields \eqref{hypA}. The second condition gives 
\[
 \int_{B_r\cap N^c_{r}} \mu_P^{\frac{2N}{N+1}}\le \sum_{y\in A} \int_{B_{r_1}(y)} \mu_P^{\frac{2N}{N+1}},
\]
concluding the proof of \eqref{conclmu}.\\

For $k\ge 0$,  we set $r_k=(5\delta)^k r$ and define recursively $A_0=\{0\}$ and,
\[
 A_k=\cup_{x\in A_{k-1}} A(x).
\]
From \eqref{hypA} we have 
\begin{equation}\label{concA}
 \sharp A_k\le (C_1 \delta^{1-N})^{k}
\end{equation}
and thus applying recursively \eqref{conclmu}, we find for $K\ge 0$,
\[
 \int_{B_r} \mu_P^{\frac{2N}{N+1}}\le C_0 \sum_{k=0}^K (\sharp A_{k}) r_k^{N-\frac{2N}{N+1}(1-\gamma)}+\sum_{y\in A_{K+1}}\int_{ B_{r_{K+1}}(y)}\mu_P^{\frac{2N}{N+1}}.
\]
By definition of $\mu_P$ we have 
\begin{multline*}
 \sum_{y\in A_{K+1}}\int_{ B_{r_{K+1}}(y)}\mu_P^{\frac{2N}{N+1}}\le (\sharp A_{K+1})| B_{r_{K+1}}| P^{\frac{2N}{N+1}} \les (C_1 \delta^{1-N})^{K} (5\delta)^{KN} r^{N} P^{\frac{2N}{N+1}}\\
 =(5^N C_1\delta)^K r^N P^{\frac{2N}{N+1}}.
\end{multline*}
Thus, if $5^N C_1\delta<1$ we can send $K\to \infty$ to obtain from the definition of $r_k$ and \eqref{concA},
\[
  \int_{B_r} \mu_P^{\frac{2N}{N+1}}\le C_0 \lt(\sum_{k\ge 0} (C_2 \delta^{1-\frac{2N}{N+1}(1-\gamma)})^{k}\rt) r^{N-\frac{2N}{N+1}(1-\gamma)},
\]
where $C_2= C_1 5^{N-\frac{2N}{N+1}(1-\gamma)}$. 
Finally if $|\gamma-\frac{1}{2}|\ll1$, $\frac{2N}{N+1}(1-\gamma)<1$ and thus provided $\delta$ is small enough, the sum converges and we have
(notice that all the constants involved are independent of $P$) 
\[
  \lt(\int_{B_r} \mu_P^{\frac{2N}{N+1}}\rt)^{\frac{N+1}{N}}\les  r^{N-1 +2\gamma}.
\]
Sending $P\to \infty$  concludes the proof of \eqref{toprovemu}.

\end{proof}
 \begin{remark}
  We point out that this estimate is essentially optimal as can be seen from the case $E=B_1$, see \cite[Chapter II.13]{landkof} and Section \ref{sec:ballmin} below.
   \end{remark} 
\begin{remark}
 A quick inspection of the proof shows that for every $q<2$, $\mu_E\in L^q_{\rm loc}(\R^N)$ if $E$ is $\delta-$Reifenberg flat with $\delta$ small enough. 
 This is again optimal in light of the boundary behavior of the $1/2-$harmonic measure of the ball see \cite[Chapter II.13]{landkof} and Section \ref{sec:ballmin} below.
 This higher integrability (with respect to $L^{\frac{2N}{N+1}}$) would  however  by itself not be sufficient to obtain the regularity of volume-constrained minimizers of $\FaQ$ so that we need the more precise estimate \eqref{toprovemu}.   
\end{remark}
Combining  Proposition \ref{prop:reg1} together with Proposition \ref{prop:am} and Lemma \ref{lem:estmu} we obtain that for small charge $Q$ every volume
constrained minimizer of $\F_{1,Q}$ is also a perimeter almost-minimizer for which the classical regularity theory applies, see \cite{maggi} so that we have the counterpart of Proposition \ref{prop:regsmall}.
\begin{proposition}\label{prop:regsmall1}
 Let $\alpha=1$. For every $\gamma\in(0,1/2)$ there exists $Q(\gamma,N)>0$ such that for every $Q\le Q(\gamma,N)$, every volume-constrained minimizer $E_Q$ of $\F_{1,Q}$ is 
 $C^{1,\gamma}$ with uniformly bounded $C^{1,\gamma}$ norm. As a consequence, for every $\beta< \gamma$, up to translation, $E_Q$ converges in $C^{1,\beta}$ to $B_1$ as $Q\to 0$. 
\end{proposition}



\section{ Rigidity of the ball for small charges }\label{sec:ballmin}
In this section we prove Theorem \ref{thm:nearlyspherical} i.e. we show that for every $\alpha\in (0,2)$ and small enough  charge $Q$, 
the ball is the unique minimizer of $\FaQ$ under volume constraints in the class of nearly spherical sets.


Before embarking in the proofs let us set some notation and make a few preliminary remarks. First, recall that fixing an arbitrary $\gamma\in(0,1)$, we say that $E$ is nearly spherical if $|E|=|B|$ (where $B=B_1$ is the unit ball), $E$ has barycenter  in $0$ and there exists $\phi:\partial B\mapsto \R$ with 
$\|\phi\|_{C^{1,\gamma}(\partial B)}\le 1$ such that 
\[
 \partial E=\{(1+\phi(x)) x \ : \ x\in \partial B\}.
\]
With a slight abuse of notation  we still denote by $\phi$ its $0-$homogeneous extension outside $\partial B$ that is, the function $\R^N\ni x\mapsto \phi(x/|x|)$.
We recall from \cite{Fuglede} that if $E$ is nearly spherical, we have 
\begin{equation}\label{meanphi}
 \lt|\int_{\partial B} \phi\rt|\les \int_{\partial B} \phi^2. 
\end{equation}
In particular, if $\|\phi\|_{W^{1,\infty}(\partial B)}\ll1$, recalling the notation  $\bar \phi=\frac{1}{P(B)}\int_{\partial B} \phi$,  we have for $s\in(0,1)$,
\begin{equation}\label{eq:Poinca}
 \int_{\partial B} \phi^2 \les \int_{\partial B} (\phi-\bar \phi)^2\stackrel{\eqref{embed}}{\les} [\phi]^2_{H^s(\partial B)}.
\end{equation}
For $\mu$ and $\nu$ two Radon measures on $\R^N$, we define the positive bilinear operator (see \cite{landkof})
\[
I_\alpha(\mu,\nu)=\int_{\R^N\times \R^N}\frac{d\mu(x)\,d\nu(y)}{|x-y|^{N-\alpha}}.
\]
In particular we have $I_\alpha(\mu)=I_\alpha(\mu,\mu)$. We let $\mu_E$  be the optimal measure of $E$ and then $u_E(x)=\int_E \frac{d\mu_E}{|x-y|^{N-\alpha}}$ 
be its associated potential. When there is no risk of confusion we drop the index $E$ from both.  
In the specific case of the unit  ball $B$ we have by  \cite[Chapter II.13]{landkof}  
\[
 \mu_B(x)=\frac{C_\alpha}{(1-|x|^2)^{\frac{\alpha}{2}}}\sim \frac{1}{d(x,\partial B)^{\frac{\alpha}{2}}}.
\]
 We sometimes use the notation $\phi_x=\phi(x)$. In particular, if $E$ is nearly spherical and $\phi$ is the corresponding parametrization, 
 we set for $x\in  B$, $T(x)=(1+\phi_x)x$ so that $E=T(B)$.  We then define $g=T^{-1}\#\mu_E$ (which is a probability measure on $B$) so that 
\begin{equation}\label{IE}
 \I_\alpha(E)=\int_{B\times B} \frac{dg_x\, dg_y}{|T(x)-T(y)|^{N-\alpha}}.
\end{equation}

We can now  begin the proofs. We first prove that $g$ has the same behavior as $\mu_B$ close to $\partial B$.
\begin{lemma}
Let $\alpha\in(0,2)$ and  $E$ be a nearly spherical set. Then its optimal measure satisfies for $x\in \partial B$
\begin{equation}\label{hypg}
 g(x)\les \frac{1}{d(x,\partial B)^{\frac{\alpha}{2}}}\sim \mu_B(x).
\end{equation}

\end{lemma}

\begin{proof}

The proof resemble the proof of Lemma \ref{lem:estmu}  taking advantage of the regularity of $E$ to obtain a sharp estimate. We first show that $\mu=\mu_E$ satisfies
\begin{equation}\label{estimmu}
 \mu(x)\les \frac{1}{d(x,\partial E)^{\frac{\alpha}{2}}}.
\end{equation}
Recall that by  \eqref{eq:ELuE} and \eqref{eq:uconstantE}, 
\[
\begin{cases}
(-\Delta)^{\frac{\alpha}{2}}u(x)=0 & x\in E^c\\
u(x)-\Ia(E)=0&x\in E.
\end{cases} 
\]
Thus, by the boundary regularity theory for the fractional Laplacian developed in   \cite{RosSer}, 
\[
u(x)-\Ia(E)\lesssim d^{\frac{\alpha}{2}}(x,\partial E).
\]
Hence, arguing  as in the proof of Lemma \ref{lem:estmu} we compute for $x\in E$,
\begin{align*}
	C'(N,\alpha)\mu(x)&\stackrel{\eqref{eq:ELuE}}{=}(-\Delta)^{\frac{\alpha}{2}}u(x)\stackrel{\eqref{LaplacePV}}{=}C(N,\alpha/2)\int_{E^c}\frac{\mathcal I_\alpha(E)-u(y)}{|x-y|^{N+\frac{\alpha}{2}}}dy\\
	&\les   \int_{E^c}\frac{ d^{\frac{\alpha}{2}}(y,\partial E)}{|x-y|^{N+\alpha}}dy\les \int_{B_{d(x,\partial E)}^c(x)} \frac{dz}{|z|^{N+\alpha-\frac{\alpha}{2}}}\\
	&\les  d(x,\partial E)^{-\frac{ \alpha}{2}}.
\end{align*}

\noindent Since $g(x)= (1+\phi_x)^N \mu( (1+\phi_x) x)$ with $|\phi|\le1/4$, up to chose $\delta$ small enough, we obtain
\[
g(x)\les\mu( (1+\phi_x) x)\les \frac{1}{d^{\frac{\alpha}{2}}((1+\phi_x) x,{\partial E})}.
\]
Thus   \eqref{hypg} follows provided  
\begin{equation}\label{claimeq}
 d( (1+\phi_x)x,\partial E)\sim |1-|x||.
\end{equation}
Let us prove \eqref{claimeq}. Since 
\[
 d( (1+\phi_x)x,\partial E)=\min_{y\in \partial B} |(1+\phi_x)x-(1+\phi_y)y|,
\]
testing with $y=\frac{x}{|x|}$ we obtain the upper bound 
\[
 d( (1+\phi_x)x,\partial E)\le (1+\phi_x)| 1-|x||\les |1-|x||.
\]
To ge the   lower bound we may assume that $|1-|x||\ll1$.  Squaring we  get 
\begin{align*}
 \lefteqn{ d^2( (1+\phi_x)x,\partial E)}\\
 &=\min_{y\in \partial B} |(1+\phi_x)x-(1+\phi_y)y|^2\\
  &=\min_{y\in \partial B} \lt\{(1+\phi_x)^2|x|^2 -2 (1+\phi_x)(1+\phi_y) x\cdot y + |1+\phi_y|^2\rt\}\\
  &=\min_{y\in \partial B} \lt\{(1+\phi_x)^2|x|^2 -2 (1+\phi_x)(1+\phi_y) |x| + |1+\phi_y|^2+ 2(1+\phi_x)(1+\phi_y)x\cdot \lt(\frac{x}{|x|}-  y\rt) \rt\}\\
  &=\min_{y\in \partial B}\lt\{ (1+\phi_x)^2 \lt||x|- \frac{1+\phi_y}{1+\phi_x}\rt|^2+ 2 (1+\phi_x)(1+\phi_y)  (|x|-y\cdot x)\rt\}\\
  &\ges \min_{y\in \partial B}\lt\{\lt||x|- 1 +\frac{\phi_x-\phi_y}{1+\phi_x}\rt|^2+ (|x|-y\cdot x)\rt\}.
\end{align*}
Now for every $y$, either $(|x|-y\cdot x)\ges ||x|-1|^2$ or $(|x|-y\cdot x)\ll | |x|-1|^2$. The first case directly leads to the conclusion of the proof of \eqref{claimeq}. 
In the second case, writing that $x=r\sigma$ with $\sigma\in \partial B$, this means that 
$|\sigma-y|^2=\frac{1}{2r}(|x|-x\cdot y)\ll ||x|-1|^2$ and thus $|\phi_x-\phi_y|\les |\sigma-y|\ll||x|-1|$ from which we find that for every $y\in \partial B$,
\[
 \lt||x|- 1 +\frac{\phi_x-\phi_y}{1+\phi_x}\rt|^2+ (|x|-y\cdot x)\ges | |x|-1|^2
\]
and the claim follows as well.
\end{proof}

Next we state and prove two  lemmas giving the Taylor expansion of the term $|T(x)-T(y)|^{-(N-\alpha)}$ appearing in \eqref{IE}.

\begin{lemma}
	For $x,y\in B$, we have 
	\begin{equation}\label{square}
	|T(x)-T(y)|^2=|x-y|^2\lt(1+\phi_x+\phi_y+\phi_x\phi_y +\psi(x,y)\rt)
	\end{equation}
	where 
	\begin{equation}\label{psi}
	\psi(x,y)= \frac{1}{2}(|x|^2+|y|^2)\lt(\frac{\phi_x-\phi_y}{|x-y|}\rt)^2+ (|x|+|y|)\lt(1-\frac{1}{2}(\phi_x+\phi_y)\rt)\frac{\phi_x-\phi_y}{|x-y|}.
	\end{equation}
	
\end{lemma}
\begin{proof}
	Expanding the squares we get 
	\begin{align*}
	|T(x)-T(y)|^2&=|(x-y)+ \frac{1}{2}( (x+y)(\phi_x-\phi_y)+ (x-y)(\phi_x+\phi_y))|^2\\
	&=|x-y|^2+ (|x|^2-|y|^2)(\phi_x-\phi_y)+|x-y|^2(\phi_x+\phi_y)\\
	&\qquad + \frac{1}{4}|x+y|^2|\phi_x-\phi_y|^2+ \frac{1}{4}|x-y|^2|\phi_x+\phi_y|^2\\
	&\qquad + \frac{1}{2}(|x|^2-|y|^2)(\phi_x^2-\phi_y^2). 
	\end{align*}
	Comparing with \eqref{square} we are left with the proof of
	\[
	\frac{1}{4}|x+y|^2|\phi_x-\phi_y|^2+ \frac{1}{4}|x-y|^2|\phi_x+\phi_y|^2=\frac{1}{2} (|x|^2+|y|^2)(\phi_x-\phi_y)^2 + |x-y|^2 \phi_x\phi_y.
	\]
	For this we write that $(\phi_x+\phi_y)^2=(\phi_x-\phi_y)^2+4 \phi_x\phi_y$ to get 
	\begin{align*}
	\frac{1}{4}|x+y|^2|\phi_x-\phi_y|^2+ \frac{1}{4}|x-y|^2|\phi_x+\phi_y|^2&= \frac{1}{4}|\phi_x-\phi_y|^2( |x+y|^2+|x-y|^2) + |x-y|^2 \phi_x\phi_y\\
	&= \frac{1}{2}  (|x|^2+|y|^2)(\phi_x-\phi_y)^2 + |x-y|^2 \phi_x\phi_y.
	\end{align*}
	
\end{proof}
As a consequence we get the following Taylor expansion of $|T(x)-T(y)|^{-(N-\alpha)}$.
\begin{lemma}\label{lem:TaylorT}
	Set $\ha=N-\alpha$. If $\|\phi\|_{W^{1,\infty}(\partial B)}\ll1$ then for $x,y\in B$, 
	\begin{equation}\label{Taylorgamma}
	|T(x)-T(y)|^{-(N-\alpha)}=|x-y|^{-(N-\alpha)}\lt((1-\frac{\ha}{2}\phi_x)(1-\frac{\ha}{2}\phi_y)-\frac{\ha}{2} \psi(x,y) + \zeta(x,y)\rt)
	\end{equation}
	where 
	\begin{equation}\label{zeta}
	|\zeta(x,y)|\les \phi_x^2+\phi_y^2+\psi^2(x,y),
	\end{equation}
	and where $\psi$ is the function defined in \eqref{psi}.
\end{lemma}
\begin{proof}
	Let us first point out that under our hypothesis we have $\|\psi\|_{L^\infty}\ll 1$. Indeed, this follows from $\|\phi\|_{W^{1,\infty}(\partial B)}\ll1$ and  
	\begin{equation}\label{difsphere}
	\frac{|x|+|y|}{|x-y|}\les \frac{1}{|\frac{x}{|x|}-\frac{y}{|y|}|} +1.
	\end{equation}
	This inequality may be easily seen using for instance polar coordinates. That is if $x=r\sigma$ and $y=s v$ then 
	\begin{multline*}
	\lt(\frac{|x|+|y|}{|x-y|}\rt)^2= \frac{ r^2+s^2}{|r-s|^2 + rs |\sigma-v|^2}= \frac{|r-s|^2}{|r-s|^2 + rs |\sigma-v|^2}+\frac{2 rs}{|r-s|^2 + rs |\sigma-v|^2}\\
	\le 1+ \frac{2}{|\sigma-v|^2}.
	\end{multline*}
	We then obtain the result by \eqref{square} and Taylor expansion.
\end{proof}

The next result contains one of the key linearization estimates we need to obtain our rigidity result.

\begin{lemma}
 Let $E$ be a nearly spherical set with $\| \phi\|_{W^{1,\infty}(\partial B)}\ll1$.  Then  for every $\alpha\in(0,2)$ 
 \begin{equation}\label{TaylorI}
  \lt|\I_\alpha(E)-I_\alpha\left( (1-\frac{\ha}{2}\phi)g\right)\rt|\les[\phi]_{H^{\frac{2-\alpha}{2}}(\partial B)}^2, 
 \end{equation}
where $(1-\frac{\ha}{2}\phi)g$ is seen as a measure on $B$ (recall that $\phi$ is extended by  $0-$homogeneity on $\R^N$) and $\ha=N-\alpha$.
\end{lemma}
\begin{proof}
 In view of \eqref{Taylorgamma} and \eqref{IE}, it is enough to prove that 
 \begin{equation}\label{toproveTaylorI}
  \lt|\int_{B\times B} \frac{\psi(x,y)}{|x-y|^{N-\alpha} } dg_x dg_y\rt|+ \lt|\int_{B\times B} \frac{\zeta(x,y)}{|x-y|^{N-\alpha} } dg_x dg_y\rt|\les [\phi]_{H^{\frac{2-\alpha}{2}}(\partial B)}^2.
 \end{equation}
Recall that from the proof of Lemma \ref{lem:TaylorT},  $\|\phi\|_{L^\infty(\partial B)}\ll1 $ implies $\|\psi\|_{L^\infty}\ll1$. 
Moreover,  by the radial symmetry of $\mu_B$ and the $0-$homogeneity of $\phi$ we have  
\begin{equation}\label{Iphi2}
 \int_{B\times B} \frac{\phi_x^2}{|x-y|^{N-\alpha}} dg_xdg_y\stackrel{\eqref{hypg}}{\les} \int_{B\times B} \phi_x^2 \frac{d\mu_B(x) d\mu_B(y)}{|x-y|^{N-\alpha}}
 =\frac{I_\alpha(B)}{\H^{N-1}(\partial B)}\int_{\partial B} \phi^2\stackrel{\eqref{eq:Poinca}}{\les}[\phi]_{H^{\frac{2-\alpha}{2}}(\partial B)}^2. 
\end{equation}
Hence, by \eqref{zeta} we are left with the proof of 
\[
 \lt|\int_{B\times B} \frac{\psi(x,y)}{|x-y|^{N-\alpha} } dg_x dg_y\rt|\les [\phi]_{H^{\frac{2-\alpha}{2}}(\partial B)}^2.
\]
By symmetry in $x$ and $y$ we have 
\[
 \int_{B\times B}  (\phi_x-\phi_y)\frac{|x|+|y|}{|x-y|^{N-\alpha} }dg_x dg_y=0.
\]
Moreover, Young inequality yields
\[
 (|x|+|y|)|\phi_x+\phi_y|\frac{|\phi_x-\phi_y|}{|x-y|}\les (|x|^2+|y|^2)\lt(\frac{\phi_x-\phi_y}{|x-y|}\rt)^2 + \phi_x^2+\phi_y^2
\]
so that by \eqref{Iphi2} and the definition \eqref{psi} of $\psi$, we just need to prove
\[
 \int_{B\times B} (|x|^2+|y|^2)\lt(\frac{\phi_x-\phi_y}{|x-y|}\rt)^2 \frac{1}{|x-y|^{N-\alpha}} dg_x dg_y\les [\phi]_{H^{\frac{2-\alpha}{2}}(\partial B)}^2.
\]
Using \eqref{difsphere} and \eqref{hypg} we have 
\begin{multline*}
{\int_{B\times B} (|x|^2+|y|^2)\lt(\frac{\phi_x-\phi_y}{|x-y|}\rt)^2 \frac{1}{|x-y|^{N-\alpha}} dg_x dg_y}\\
 \les \int_{B\times B}\lt[ \lt(\frac{\phi_x-\phi_y}{|\frac{x}{|x|}-\frac{y}{|y|}|}\rt)^2+ (\phi_x-\phi_y)^2\rt]\frac{ d\mu_B(x)d\mu_B(y)}{|x-y|^{N-\alpha}}\\
 \les\int_{B\times B} \lt(\frac{\phi_x-\phi_y}{|\frac{x}{|x|}-\frac{y}{|y|}|}\rt)^2\frac{ d\mu_B(x)d\mu_B(y)}{|x-y|^{N-\alpha}} + \int_{\partial B} \phi^2,
\end{multline*}
where we used Young inequality and formula \eqref{Iphi2} to estimate the second term in the last inequality. Thanks to \eqref{eq:Poinca},  we may further reduce the proof of  \eqref{TaylorI} to 
\begin{equation}\label{reduction}
 \int_{B\times B} \lt(\frac{\phi_x-\phi_y}{|\frac{x}{|x|}-\frac{y}{|y|}|}\rt)^2\frac{ d\mu_B(x)d\mu_B(y)}{|x-y|^{N-\alpha}} \les[\phi]_{H^{\frac{2-\alpha}{2}}(\partial B)}^2. 
\end{equation}
Recalling that $\mu_B(x)\les (1-|x|)^{-\frac{\alpha}{2}}$ and writing $x$ and $y$ in polar coordinates $x=r\sigma$ and $y=sv$, with $r,s\in\R$ and $\sigma, v\in\partial B$, we get
\begin{multline*}
 \int_{B\times B} \lt(\frac{\phi_x-\phi_y}{|\frac{x}{|x|}-\frac{y}{|y|}|}\rt)^2\frac{ d\mu_B(x)d\mu_B(y)}{|x-y|^{N-\alpha}}\\
 \les\int_{\partial B\times \partial B} \lt(\frac{\phi(\sigma)-\phi(v)}{|\sigma-v|}\rt)^2\lt[\int_0^1\int_0^1 \frac{r^{N-1}s^{N-1}}{|1-r|^{\frac{\alpha}{2}}|1-s|^{\frac{\alpha}{2}}} \frac{dr ds}{(|r-s|^2 +rs |\sigma-v|^2)^{\frac{N-\alpha}{2}}}\rt] d\sigma dv\\
 =\int_{\partial B\times \partial B} \lt(\frac{\phi(\sigma)-\phi(v)}{|\sigma-v|}\rt)^2 F(|\sigma-v|) d\sigma dv
\end{multline*}
where 
\[
 F(\theta)=\int_0^1\int_0^1 \frac{r^{N-1}s^{N-1}}{|1-r|^{\frac{\alpha}{2}}|1-s|^{\frac{\alpha}{2}}} \frac{dr ds}{(|r-s|^2 +rs \theta^2)^{\frac{N-\alpha}{2}}}.
\]
We claim that for $0<\theta\le 2$,
\begin{equation}\label{eq:claimF}
 F(\theta)\les \frac{1}{\theta^{N-\alpha-1}}.
\end{equation}
It is enough to prove this estimate for $\theta\ll1$. To this aim we first estimate
\begin{align*}
 \int_0^{\frac{1}{2}}\int_0^1 \frac{r^{N-1}s^{N-1}}{|1-r|^{\frac{\alpha}{2}}|1-s|^{\frac{\alpha}{2}}} \frac{dr ds}{(|r-s|^2 +rs \theta^2)^{\frac{N-\alpha}{2}}}&\les \int_0^{\frac{1}{2}} r^{N-1}\lt[\int_0^1 \frac{1}{|1-s|^{\frac{\alpha}{2}}} \frac{ds}{(|r-s|^2 +rs \theta^2)^{\frac{N-\alpha}{2}}}\rt] dr\\
 &\les \int_0^{\frac{1}{2}} r^{N-1}\lt[\int_0^{\frac{3}{4}} \frac{ds}{(|r-s|^2 +rs \theta^2)^{\frac{N-\alpha}{2}}}\rt] dr\\
 &\qquad +\int_0^{\frac{1}{2}} \lt[\int_{\frac{3}{4}}^1 \frac{ds}{|1-s|^{\frac{\alpha}{2}}}\rt] dr\\
 &\les \int_0^{\frac{1}{2}} r^{N-1}\lt[\int_{\R} \frac{dt}{(t^2 +r^2 \theta^2 +r t\theta^2)^{\frac{N-\alpha}{2}}}\rt] dr +1,
\end{align*}
where in the last line we did the change of variables $s=r+t$. We now write that by the change of variables $t= r\theta s$,
\begin{multline*}
 \int_{\R} \frac{dt}{(t^2 +r^2 \theta^2 +r t\theta^2)^{\frac{N-\alpha}{2}}}= \frac{1}{(r\theta)^{N-\alpha-1}}\int_{\R} \frac{ds}{(s^2 +1+\theta s)^{\frac{N-\alpha}{2}}}\\
 \les  \frac{1}{(r\theta)^{N-\alpha-1}}\int_{\R} \frac{ds}{((s+1)^2+|s|)^{\frac{N-\alpha}{2}}}\les \frac{1}{(r\theta)^{N-\alpha-1}},
\end{multline*}
where we used that since $\theta\le 1$, $s^2 +1+\theta s\ges (s+1)^2 +|s|$. We thus conclude that 
\begin{equation}\label{eq:integ1}
 \int_0^{\frac{1}{2}}\int_0^1 \frac{r^{N-1}s^{N-1}}{|1-r|^{\frac{\alpha}{2}}|1-s|^{\frac{\alpha}{2}}} \frac{dr ds}{(|r-s|^2 +rs \theta^2)^{\frac{N-\alpha}{2}}}\les 1+\frac{1}{\theta^{N-\alpha-1}}\int_0^{\frac{1}{2}} r^\alpha dr\les \frac{1}{\theta^{N-\alpha-1}}.
\end{equation}
We now focus on the integral between $1/2$ and $1$ which we split as 
\begin{align*}
 \int_{\frac{1}{2}}^1\int_0^1 \frac{r^{N-1}s^{N-1}}{|1-r|^{\frac{\alpha}{2}}|1-s|^{\frac{\alpha}{2}}} \frac{dr ds}{(|r-s|^2 +rs \theta^2)^{\frac{N-\alpha}{2}}}&\les\int_{\frac{1}{2}}^1\int_0^{\frac{1}{4}} \frac{1}{|1-r|^{\frac{\alpha}{2}}}dr ds\\
 &\qquad +\int_{\frac{1}{2}}^1\int_{\frac{1}{4}}^1 \frac{1}{|1-r|^{\frac{\alpha}{2}}|1-s|^{\frac{\alpha}{2}}} \frac{dr ds}{(|r-s|^2 + \theta^2)^{\frac{N-\alpha}{2}}}\\
 &\les 1+\int_0^{\frac{1}{2}} \int_{-1}^1 \frac{1}{t^{\frac{\alpha}{2}} |t-w|^{\frac{\alpha}{2}}} \frac{dt dw}{(w^2+\theta^2)^{\frac{N-\alpha}{2}}},
\end{align*}
where in the last line we made the change of variables $r=1-t$ and $s=1-t+w$. Now for every $w\in (-1,1)$, 
\[
 \int_0^{\frac{1}{2}}\frac{dt}{t^{\frac{\alpha}{2}} |t-w|^{\frac{\alpha}{2}}}\le  \int_0^{\frac{1}{2}}\frac{dt}{t^\alpha}+ \int_0^{\frac{1}{2}}\frac{dt}{|t-w|^{\alpha}}\les 1
\]
and thus 
\[
 \int_0^{\frac{1}{2}} \int_{-1}^1 \frac{1}{t^{\frac{\alpha}{2}} |t-w|^{\frac{\alpha}{2}}} \frac{dt dw}{(w^2+\theta^2)^{\frac{N-\alpha}{2}}}\les 
 \int_{-1}^1 \frac{ dw}{(w+\theta)^{N-\alpha}}\les \frac{1}{\theta^{N-\alpha-1}}.
\]
This proves 
\[
 \int_{\frac{1}{2}}^1\int_0^1 \frac{r^{N-1}s^{N-1}}{|1-r|^{\frac{\alpha}{2}}|1-s|^{\frac{\alpha}{2}}} \frac{dr ds}{(|r-s|^2 +rs \theta^2)^{\frac{N-\alpha}{2}}}\les \frac{1}{\theta^{N-\alpha-1}},
\]
which together with \eqref{eq:integ1} concludes the proof of \eqref{eq:claimF}.
We thus find
\[
  \int_{B\times B} \lt(\frac{\phi_x-\phi_y}{|\frac{x}{|x|}-\frac{y}{|y|}|}\rt)^2\frac{ d\mu_B(x)d\mu_B(y)}{|x-y|^{N-\alpha}}
  \les \int_{\partial B\times \partial B} \frac{(\phi(\sigma)-\phi(v))^2}{|\sigma-v|^{N-\alpha +1}} d\sigma dv\stackrel{\eqref{HsSphere}}{=} [\phi]_{H^{\frac{2-\alpha}{2}}(\partial B)}^2,
\]
which is \eqref{reduction}.
\end{proof}
%

We may now conclude the proof of the stability inequality for nearly spherical sets. 
\begin{proposition}\label{prop:comparison}
 If $E$ is a nearly spherical set with $\|\phi\|_{W^{1,\infty}(\partial B)}\ll1$, then for $\alpha \in (0,2)$,
 \begin{equation}\label{quantI}
  \I_\alpha(B)-\I_\alpha(E)\les [\phi]^2_{H^{\frac{\alpha}{2}}(\partial B)}+[\phi]_{H^{\frac{2-\alpha}{2}}(\partial B)}^2.
 \end{equation}
 As a consequence,
  \begin{equation}\label{quantI2}
  \I_\alpha(B)-\I_\alpha(E)\les P(E)-P(B).
 \end{equation}
\end{proposition}
\begin{proof}
 Using the same notation as above and using that $I_\alpha(g)=I_\alpha(g-\mu_B)+2I_\alpha(g-\mu_B,\mu_B)+ I_\alpha(\mu_B)$, we have 
 \begin{align*}
  \I_\alpha(B)-\I_\alpha(E)&=I_\alpha(\mu_B)-\I_\alpha(E)\\
  &= I_\alpha(\mu_B)-I_\alpha(g) + I_\alpha(g)-\I_\alpha(E)\\
  &= -I_\alpha(g-\mu_B)-2 I_\alpha(g-\mu_B,\mu_B)+ I_\alpha(g)-\I_\alpha(E).
 \end{align*}
We now notice that by optimality of $\mu_B$ we have that $u_B$ is constant in $B$ (recall \eqref{eq:uconstantE}) and thus, since $\int_B \mu_B=\int_B g=1$,
\[
I_\alpha(g-\mu_B,\mu_B)=\int_B u_B (g-\mu_B)= u_B(0)\int_B (g-\mu_B)=0. 
\]
Using \eqref{TaylorI} we can compute  
\begin{align*}
 \I_\alpha(B)-\I_\alpha(E)+I_\alpha(g-\mu_B)&\le I_\alpha(g)-I_\alpha((1-\frac{\ha}{2}\phi)g) + C [\phi]_{H^{\frac{2-\alpha}{2}}(\partial B)}^2\\
 &=-\frac{\ha^2}{4}I_\alpha(\phi g)+\ha I_\alpha(g,\phi g) + C [\phi]_{H^{\frac{2-\alpha}{2}}(\partial B)}^2\\
 &\les I_\alpha(g,\phi g) +[\phi]_{H^{\frac{2-\alpha}{2}}(\partial B)}^2.
\end{align*}
We further decompose the term $I_\alpha(g,\phi g)$ as follows:
\begin{align*}
 I_\alpha(g,\phi g)&=I_\alpha(\mu_B,\phi g)+ I_\alpha(g-\mu_B,\phi g)\\
 &=I_\alpha(\mu_B, \phi \mu_B)+ I_\alpha(\mu_B, \phi(g-\mu_B))+ I_\alpha(g-\mu_B,\phi g).
\end{align*}
We now observe that since $\mu_B$ is radially symmetric and since $\phi$ is $0-$homogeneous,
\[
 I_\alpha(\mu_B, \phi \mu_B)=C \int_{\partial B} \phi\stackrel{\eqref{meanphi}}{\les}\int_{\partial B}  \phi^2.
\]
By \eqref{eq:Poinca}, we therefore have 
\begin{equation}\label{almostdone}
 \I_\alpha(B)-\I_\alpha(E)+I_\alpha(g-\mu_B)\les  I_\alpha(\mu_B, \phi(g-\mu_B))+   I_\alpha(g-\mu_B,\phi g)+[\phi]_{H^{\frac{2-\alpha}{2}}(\partial B)}^2.
\end{equation}

\noindent We first estimate $I_\alpha(g-\mu_B,\phi g)$. We notice that 
\[
 I_\alpha(\phi g)\le \lt(\int_{B\times B} \frac{\phi_x^2 g_x g_y}{|x-y|^{N-\alpha}}\rt)^{\frac{1}{2}}\lt(\int_{B\times B} \frac{\phi_y^2 g_x g_y}{|x-y|^{N-\alpha}}\rt)^{\frac{1}{2}}
 \stackrel{\eqref{Iphi2}}{\les} \int_{\partial B} \phi^2\stackrel{\eqref{eq:Poinca}}{\les}  [\phi]_{H^{\frac{\alpha}{2}}(\partial B)}^2.
\]
Thus, Cauchy-Schwarz  inequality for $I_\alpha$ (recall that it is a positive bilinear operator) gives 
\begin{equation}\label{almostdone1}
 I_\alpha(g-\mu_B,\phi g)\le I^{\frac{1}{2}}_\alpha(g-\mu_B) I^{\frac{1}{2}}_\alpha(\phi g)\les I^{\frac{1}{2}}_\alpha(g-\mu_B) [\phi]_{H^{\frac{\alpha}{2}}(\partial B)}.
\end{equation}

%

%
\noindent  We now turn to  $I_\alpha(\mu_B, \phi(g-\mu_B))$. For this we use that $u_B$ is constant on $B$ to write 
\[
 I_\alpha(\mu_B, \phi(g-\mu_B))=u_B(0)\int_B \phi (g-\mu_B).
\]
Let  $\rho$ be a smooth, positive cut-off function with $\rho=1$ on $B$ and $\rho=0$ on $B_2^c$. We then set $\Phi=\phi\rho$ so that 
\begin{align*}
 \int_B \phi (g-\mu_B)&=\int_{\R^N} \Phi (g-\mu_B)\\
 &\le [\Phi]_{H^{\frac{\alpha}{2}}(\R^N)} [g-\mu_B]_{H^{-\frac{\alpha}{2}}(\R^N)}\\
 &\stackrel{\eqref{IH}}{\les}[\Phi]_{H^{\frac{\alpha}{2}}(\R^N)} I_\alpha^{\frac{1}{2}}(g-\mu_B).
\end{align*}
We finally show that  
\begin{equation}\label{boundary}
[\Phi]_{H^{\frac{\alpha}{2}}(\R^N)}^2\les [\phi]^2_{H^{\frac{\alpha}{2}}(\partial B)}+ \int_{\partial B} \phi^2.
\end{equation}
For every $x,y$,
\begin{multline*}
 (\Phi_x-\Phi_y)^2=(\phi_x \rho_x-\phi_y\rho_y)^2\les (\phi_x-\phi_y)^2 \rho_x^2+ \rho_y^2 (\rho_x-\rho_y)^2
 \les (\phi_x-\phi_y)^2+ \phi_y^2(x-y)^2,
\end{multline*}
so that 
\begin{multline*}
  [\Phi]_{H^{\frac{\alpha}{2}}(\R^N)}^2\stackrel{\eqref{Hsbis}}{\les} \int_{B_2\times B_2}\frac{(\Phi_x-\Phi_y)^2}{|x-y|^{N+\alpha}}\les
  \int_{B_2\times B_2}\frac{(\phi_x-\phi_y)^2}{|x-y|^{N+\alpha}}+  \int_{B_2\times B_2}\frac{\phi_y^2}{|x-y|^{N+\alpha-2}}\\
  \les \int_{B_2\times B_2}\frac{(\phi_x-\phi_y)^2}{|x-y|^{N+\alpha}}+\int_{\partial B} \phi^2.
\end{multline*}
Using polar coordinates we now write
\[
 \int_{B_2\times B_2}\frac{(\phi_x-\phi_y)^2}{|x-y|^{N+\alpha}}=\int_{\partial B\times \partial B} (\phi(\sigma)-\phi(v))^2\lt[\int_0^2\int_0^2 r^{N-1} s^{N-1} \frac{dr ds}{((r-s)^2+ rs |\sigma-v|^2)^{\frac{N+\alpha}{2}}}\rt] d\sigma dv.
\]
Arguing as for \eqref{eq:claimF} we have 
\[
 \int_0^2\int_0^2 r^{N-1} s^{N-1} \frac{dr ds}{((r-s)^2+ rs |\sigma-v|^2)^{\frac{N+\alpha}{2}}}\les \frac{1}{|\sigma-v|^{N-1+\alpha}},
\]
which concludes the proof of \eqref{boundary}. Recalling \eqref{eq:Poinca} we find 
\begin{equation}\label{almostdone2}
 I_\alpha(\mu_B, \phi(g-\mu_B))\les I_\alpha^{\frac{1}{2}}(g-\mu_B)[\phi]_{H^{\frac{\alpha}{2}}(\partial B)}.
\end{equation}
Plugging  \eqref{almostdone1} and \eqref{almostdone2} into \eqref{almostdone} we get 
\[
  \I_\alpha(B)-\I_\alpha(E)+I_\alpha(g-\mu_B)\les   I_\alpha^{\frac{1}{2}}(g-\mu_B)[\phi]_{H^{\frac{\alpha}{2}}(\partial B)}+[\phi]_{H^{\frac{2-\alpha}{2}}(\partial B)}^2.
\]
Using Young inequality we conclude the proof of \eqref{quantI}.\\

Since $E$ is nearly spherical we have\footnote{this is the only place where we use that the barycenter of $E$ is in $0$.} (see \cite{Fuglede})
\[
 \int_{\partial B} |\nabla \phi|^2\les P(E)-P(B)
\]
so that \eqref{quantI2} follows using \eqref{embed}.
\end{proof}
We can now conclude the proof of Theorem \ref{thm:nearlyspherical}.
\begin{proof}[Proof of Theorem \ref{thm:nearlyspherical}]
Let $E$ be a nearly spherical set with $\|\phi\|_{W^{1,\infty}(\partial B)}\ll1$. If $\FaQ(E)\le \FaQ(B)$, then rearranging terms we find
\[
 P(E)-P(B)\le Q^2\lt(\Ia(B)-\Ia(E)\rt)\stackrel{\eqref{quantI2}}{\les} Q^2 \lt(P(E)-P(B)\rt).
\]
This implies that either $P(E)=P(B)$ and thus $E=B$ by the isoperimetric inequality or $1\les Q^2$ which proves the claim. 
\end{proof}

\section{Non existence in dimension $2$}\label{sec:nonex}
We show here a nonexistence result in dimension  $ 2$. Namely, that in $N=2$  minimizers in $\mathcal S$ (and hence classical minimizers) cannot exist for large $Q$.
\begin{theorem}\label{thm:nonex}
 Let $N=2$ and $\alpha\in(0,1]$. Then, for $Q\gg1$ there are no  minimizers of 
 \begin{equation}\label{eq:prob}
  \min \lt\{ \FaQ(E) \ : \ |E|=\omega_N,\, E\in\mathcal S\rt\}.
 \end{equation}

\end{theorem}
\begin{proof}
Let us point out that although the case $\alpha=1$ is already covered by \cite{murnovruf} (with an explicit threshold between existence and non-existence) we will still include it in the proof. 
We follow the ideas of \cite[Theorem 3.3]{KnuMu}  in the streamlined version of \cite{FKM}. 
For $\nu\in \partial B_1$ and $t\in \R$, we let 
\[
 H_{\nu,t}^+=\{ x\cdot\nu \ge t\},  \qquad H_{\nu,t}^-=\{ x\cdot\nu < t\} \qquad \textrm{and } \qquad H_{\nu,t}=\{ x\cdot\nu = t\}.  
\]
We then define for any measure $\mu$ and set $E$,
\[
 \mu_{\nu,t}^{\pm}= \mu|_{H_{\nu,t}^\pm}  \qquad \textrm{and}   \qquad   E_{\nu,t}^{\pm}= E\cap H_{\nu,t}^\pm.
\]
Assume that $E$ is a minimizer of \eqref{eq:prob}. Comparing the energy of $E$ with the one of two infinitely far apart copies of $E^\pm_{\nu,t}$ with measures $\mu^{\pm}_{\nu,t}$, we have 
\[
 \FaQ(E)\le P(E^+_{\nu,t}) +P(E^-_{\nu,t})+ Q^2 I_\alpha(\mu^+_{\nu,t})+Q^2I_\alpha(\mu^{-}_{\nu,t}).
\]
Using that $P(E)=P(E^+_{\nu,t}) +P(E^-_{\nu,t})-2 \H^1(E\cap H_{\nu, t})$ and $I_\alpha(E)=I_\alpha(\mu^+_{\nu,t})+I_\alpha(\mu^{-}_{\nu,t})+2I_\alpha(\mu^+_{\nu,t},\mu^-_{\nu,t})$, this simplifies to 
\[
 Q^2 I_\alpha(\mu^+_{\nu,t},\mu^-_{\nu,t})\le \H^1(E\cap H_{\nu, t}).
\]
We now integrate this inequality in $t$ and $\nu$ to get
\begin{align*}
|E|&\ges \int_{\partial B_1}\int_\R \H^1(E\cap H_{\nu, t})\,d\nu \\
 &\ge  Q^2\int_{\partial B_1}\int_\R \Ia(\mu^+_{\nu,t},\mu^-_{\nu,t})\,d\nu\\
 &=Q^2\int_{\partial B_1}\int_\R \int_{H^+_{\nu,t}\times H^-_{\nu,t} } \frac{d\mu(x)\, d\mu(y)}{|x-y|^{2-\alpha}}\,d\nu \\
 &\ges Q^2 \int_{\R^2\times\R^2} \frac{d\mu(x) \,d\mu(y)}{|x-y|^{1-\alpha}},
\end{align*}
where we used that for every $(x,y)$,
\[
 \int_{\partial B_1}\int_\R \chi_{H^+_{\nu,t}\times H^-_{\nu,t}}(x,y) dt\, d\nu\sim |x-y|.
\]
Since $|E|=\omega_N$, this yields the estimate 
\[
 1\ges \frac{Q^2}{d^{1-\alpha}} 
\]
where $d={\rm diam}(E)$. If $\alpha=1$ this already gives the conclusion so that we are left with the case $\alpha<1$. Since for $N=2$, $P(E)\ges d$, we get the lower bound 
\[
 \FaQ(E)\ges  Q^{\frac{2}{1-\alpha}}.
\]
 For  a generalized set $\wE_r$ made of $n$ copies of the ball of radius $r= n^{-1/2}$, we have
 \[
  \FaQ(\wE_r)\les n r +\frac{Q^2}{n r^{2-\alpha}}=r^{-1}+Q^2 r^\alpha. 
 \]
Optimizing in $r$ by choosing $r=Q^{-\frac{2}{1+\alpha}}$, we find by minimality of $E$,
\[
 Q^{\frac{2}{1-\alpha}}\les \FaQ(E)\le \FaQ(\wE_r)\les Q^{\frac{2}{1+\alpha}}, 
\]
which is absurd if $Q\gg1$.
\end{proof}
\begin{remark}
 While we believe that the same result holds for  $N\ge 3$, it is well-known that this kind of arguments gives useful information only when $\alpha>N-2$, 
 which is compatible with $\alpha\le 1$ only if $N=2$.
\end{remark}

\noindent
{\bf Acknowledgments.}
We thank A. Zilio for useful discussions about the Alt-Caffarelli-Friedman monotonicity formula. M. Goldman and B. Ruffini  were partially supported by the project ANR-18-CE40-0013 SHAPO financed by the French Agence Nationale de la Recherche (ANR), and M. Novaga was partially supported by the PRIN Project 2019/24.
B. Ruffini and M. Novaga are members of the INDAM-GNAMPA.


\bibliographystyle{acm}

\end{document}